\documentclass[11pt]{amsart}
\usepackage{amsmath,amsfonts}
\usepackage{amssymb}
\usepackage{amscd}
\usepackage{amsthm}
\usepackage{yhmath}
\usepackage{subfigure}
\usepackage[all]{xy}
\usepackage{color}

\usepackage{pdfsync}

\def\CB{\color{black}}

\numberwithin{equation}{section}

\theoremstyle{plain}
\newtheorem{theorem}{Theorem}[section]
\newtheorem*{theorem*}{Theorem}
\newtheorem{lemma}[theorem]{Lemma}
\newtheorem{corollary}[theorem]{Corollary}
\newtheorem{proposition}[theorem]{Proposition}
\theoremstyle{remark}
\newtheorem{remark}[theorem]{Remark}
\numberwithin{equation}{section}
\theoremstyle{definition}
\newtheorem{definition}[theorem]{Definition}
\numberwithin{equation}{section}

\def\cR{{\mathcal R}}

\newcommand{\bbR}{{\mathbb{R}}}

\newcommand{\CC}{\mathbb{C}}
\newcommand{\di}{\,{\rm{d}}}
\newcommand{\R}{\mathbb R}

\begin{document}

\title{Sobolev algebras on nonunimodular Lie groups}


\author[M. M. Peloso]{Marco M. Peloso}
\address{
Dipartimento di Matematica, 
Universit\`a degli Studi di Milano, 
Via C. Saldini 50,  
20133 Milano, Italy}
\email{marco.peloso@unimi.it}

\author[M. Vallarino]{Maria Vallarino}
\address{Dipartimento di Scienze Matematiche ``Giuseppe Luigi Lagrange", Dipartimento di Eccellenza 2018-2022, 
  Politecnico di Torino, Corso Duca degli Abruzzi 24, 10129 Torino,
  Italy} 
\email{maria.vallarino@polito.it}


\thanks{Both authors are  partially supported by the grants   PRIN
2010-11 and 
  2015 {\em Real and Complex Manifolds: Geometry, Topology and
    Harmonic Analysis}. Both authors are members of the 
Gruppo Nazionale per l'Analisi Matematica, la Probabilit\`a e le loro Applicazioni
(GNAMPA) of the Istituto Nazionale di Alta Matematica (INdAM)}
 
\begin{abstract}
Let $G$ be a noncompact connected Lie group and $\rho$ be the right Haar measure of $G$.
Let ${\bf{X}}=\{X_1,\dots,X_q\}$ be a family of left invariant
vector fields which satisfy H\"ormander's condition, and let
 $\Delta=-\sum_{i=1}^qX_i^2$ be the corresponding
 subLaplacian.
 For $1\leq
p<\infty$ and $\alpha\geq 0$ we define the Sobolev space
$$
L^p_{\alpha}(G)=\{f\in L^p(\rho): \Delta^{\alpha/2}f\in L^p(\rho)\}\,,
$$
endowed with the norm
$$
\|f\|_{\alpha,p}=\|f\|_{p}+\|\Delta^{\alpha/2}f\|_p\,,
$$ 
where we denote by $\|f\|_p$ the norm of $f$ in
$L^p(\rho)$.  

In this paper  we show that for all $\alpha\geq 0$ and $p\in
(1,\infty)$, the space 
$L^{\infty}\cap L^p_{\alpha}(G)$ is an algebra under pointwise
product, that is, 
there exists a positive constant
  $C_{\alpha,
  p}$ such that for all $f,g\in L^{\infty}\cap L^p_{\alpha}(G)$, $fg\in
L^{\infty}\cap L^p_{\alpha}(G)$ and 
$$
\|fg\|_{\alpha,p}\leq C_{\alpha,p}
\big(\|f\|_{\alpha,p}\|g\|_{\infty}+\|f\|_{\infty}\|g\|_{\alpha,p}\big)\,.
$$

Such estimates  were proved by T. Coulhon, E. Russ and
V. Tardivel-Nachef 
in the case
when $G$ is unimodular. We shall prove it on
Lie groups, thus extending their result to the  nonunimodular case. 

In order to prove our main result, 
we need to study the boundedness of local Riesz transforms  
$R^c_J=X_J(cI+\Delta)^{-m/2}$, where $c>0$, $X_J=X_{j_1}\dots X_{j_m}$ and
$j_\ell \in\{1,\dots,q\}$ for $\ell=1,\dots,m$.  
 We show that if $c$ is 
sufficiently large,   the Riesz
transform $R^c_J$ is bounded on $L^p(\rho)$ for every $p\in (1,\infty)$, and prove
also appropriate endpoint results involving Hardy and BMO spaces.
\end{abstract}

\maketitle 

\section{Introduction and statement of the main results}

Let $G$ be a noncompact connected Lie group. We shall denote by
$\lambda$ and $\rho$ the left and right Haar measures of $G$,
respectively, and by $\delta$ the modular function,
i.e. $\delta=\frac{\di\lambda}{\di\rho}$. For every $p\in [1,\infty]$
and $f\in L^p(\rho)$ we shall denote by $\|f\|_p$ the norm of $f$ in
$L^p(\rho)$.

Let ${\bf{X}}=\{X_1,\dots,X_q\}$ be a family of left invariant
vector fields which satisfy H\"ormander's condition and
consider the subLaplacian $\Delta=-\sum_{i=1}^qX_i^2$. For every
$ p\in (1,\infty)$, let $\Delta_p$ be the smallest closed extension of
$\Delta |_{C^{\infty}_c(G)}$ to $L^p(\rho)$. For every $\alpha>0$
one may define the operator $\Delta^{\alpha}_p$ on $L^p(\rho)$ which
we shall always denote by $\Delta^{\alpha}$, see
e.g. \cite{K1}. For every $p\in (1,\infty)$ and $\alpha\geq 0$ we define the Sobolev space
$$
L^p_\alpha(G)
=\big\{f\in L^p(\rho): \Delta^{\alpha/2}f\in L^p(\rho)\big\}\,,
$$
endowed with the norm
\begin{equation}\label{Sobolev-norm}
\|f\|_{\alpha,p}=\|f\|_{p}+\|\Delta^{\alpha/2}f\|_p\,. \medskip
\end{equation}

Throughout 
the  paper we will often denote by $L^p$ the space
$L^p(\rho)$, and when we refer to $L^p$-integrability, we will always 
mean integrability with respect to the right Haar measure $\di \rho$. 
Moreover, if the underlying group $G$ is understood from the context,
we will often write $L^p_\alpha$ in place of $L^p_\alpha(G)$.

Our aim is to prove the following result. 
\begin{theorem}\label{t: productlemma}
Let $G$ be a noncompact connected Lie group. For all $\alpha\geq 0$ and $p\in (1,\infty)$ the space
$L^p_{\alpha}\cap L^{\infty}$ is an algebra under pointwise
product. More precisely, there exists a positive constant  $C_{\alpha,
  p}$ such that for all $f,g\in L^p_{\alpha}\cap L^{\infty}$, we have $fg\in
L^p_{\alpha}\cap L^{\infty}$ and 
$$
\|fg\|_{\alpha,p}\leq C_{\alpha,p}
\big(\|f\|_{\alpha,p}\|g\|_{\infty}+\|f\|_{\infty}\|g\|_{\alpha,p}\big)\,.
$$
\end{theorem} 
Theorem \ref{t: productlemma} will be obtained as a particular case of
the following more general theorem. 
\begin{theorem}\label{t: productlemma2}
Let $G$ be a noncompact connected Lie group. Let $\alpha\geq 0$,
$p_1,q_2\in (1,\infty]$ and $r,p_2,q_1\in 
(1,\infty)$ such that $\frac1r=\frac{1}{p_i}+\frac{1}{q_i}$, $i=1,2$. 
There exists a positive constant $C$ such that for all $f\in
L^{p_1}(\rho)\cap L^{p_2}_{\alpha}$ 
and $g\in L^{q_2}(\rho)\cap L^{q_1}_{\alpha}$, we have $fg\in   L^r_{\alpha}$ and
%\begin{equation}\label{main-alg-ests}
$$
\|fg\|_{\alpha,r}\leq C 
\big(\|f\|_{p_1}\|g\|_{\alpha,q_1}+\|f\|_{\alpha, p_2}\|g\|_{q_2}\big)\,.
$$
%\end{equation}\medskip
\end{theorem} 

Given a Laplacian or a subLaplacian on a Lie group, the question of finding
under which conditions the corresponding Sobolev spaces form an algebra,
has
 a long history.  It was first proved by R. Strichartz \cite{Strichartz-multi}
in the case of the Laplacian in $\bbR^n$ that the Sobolev spaces
$L^p_\alpha(\bbR^n)$ form an algebra when $\alpha p>n$.  Such result
was later extended 
by G. Bohnke in the case
of a nilpotent Lie group  $G$  \cite{B}, under the condition   $\alpha p>Q$,
where $Q$ denotes the homogeneous dimension of $G$.  This result was
also proved by T. Coulhon, E. Russ and V. Tardivel-Nachev \cite{CRTN}
on any unimodular Lie group $G$ when $\alpha p>d$, where $d$ denotes
the local dimension of $G$; see
\eqref{pallepiccole}. 

Later, T. Kato and G. Ponce
 \cite{KaTo} 
proved Theorem \ref{t: productlemma} in the case of the Laplacian in
$\bbR^n$, which is more general than 
Strichartz's result since 
it does not rely on the Sobolev embedding.
Incidentally, the same authors showed that
the  algebra property of the Sobolev spaces is
 fundamental in the theory of well-posedness of Cauchy problems for
 certain nonlinear differential equations.

More recently, the Sobolev algebra problem was studied for Laplacians,
subLaplacians 
and even more general differential operators satisfying suitable
assumptions on various Lie groups and Riemannian manifolds
\cite{BBR,BCF,CRTN,GS}. In particular, 
Theorems \ref{t: productlemma} and \ref{t:
  productlemma2} 
were proved in \cite{CRTN} in the case
when $G$ is unimodular.  

As already mentioned, in this paper we prove Theorems \ref{t: productlemma} and \ref{t:
  productlemma2} in the case of a subLaplacian on any nonunimodular Lie group.  The
 situation on a nonunimodular Lie group is considerably more
 complicated than in the unimodular case. Indeed we prove that in
 general when $G$ is nonunimodular,  
 $L^p_\alpha(G)$ is not an algebra, even when $\alpha p>d$, see
 Theorem \ref{contro-esempio}.  
 Incidentally, the same counterexample shows
 that the space $L^p_1(G)$ does not embed in $L^\infty$ when $p>d$;
 see also Remark \ref{contro-eempio-rem}. 
Furthermore, 
since $\delta$ is not trivial, we have to deal
with some technical difficulties: $\delta$ obviously appears when we
make some change of variables in the integrals, and the factor
$\delta^{1/2}$ naturally arises in the estimates of the heat kernel
associated with $\Delta$ and its derivatives (see Subsection \ref{pt}
below). Very often we shall work on balls of small radius where the
modular function is comparable with its value at the center of the
ball; but sometimes we also have to deal with the behavior of the
modular function on balls of arbitrary radius. Let us mention that
property \eqref{intdelta12} below, which gives a control of the
integral of $\delta^{1/2}$ on balls of any radius, is crucial in the
proof of our results.   
 
The main ingredient in the proof of Theorem \ref{t: productlemma}, is
a characterization of the Sobolev norm \eqref{Sobolev-norm} in 
terms of the $p$-integrability property of averages of differences of
a function on small balls.  We consider the local versions  of 
functionals introduced by Strichartz \cite{Strichartz-multi}, and
 E. M.  
Stein \cite{Stein-singular-integrals}, respectively; see
also \cite{CRTN} for these local versions in the unimodular case. 
To be more precise, for  a locally integrable function 
$f$ and every $\alpha\in (0,1)$ 
we set
\begin{equation}\label{S-alpha-loc-def}
S^{{\rm loc}}_{\alpha}f(x)
=\Big(\int_0^1\Big[ \frac{1}{u^{\alpha}
  V(u)}\int_{|y|<u}|f(xy^{-1})-f(x)|\di\rho(y) \Big]^2\frac{\di u}{u}
\Big)^{1/2}\,,
\end{equation}
and
\begin{equation}\label{D-alpha-loc-def}
D^{{\rm loc}}_{\alpha}f(x)=\Big( \int_{|y|<1}      \frac{
  |f(xy^{-1})-f(x)|^2  }{|y|^{2\alpha} V(|y|)}     \di\rho (y)
\Big)^{1/2}\,. 
\end{equation}
For $r>0$, we denote by $V(r)$ the volume of the ball centered at
the origin $e$ of $G$, with respect to the right Haar measure $\rho$;
see \eqref{Vr}. Then, we prove the following result. 
\begin{theorem}\label{Salphaloc} 
Let $G$ be a noncompact connected Lie group and let $\alpha\in (0,1)$. 
 Then the  following properties hold: 
\begin{itemize}
\item[(i)]   for any $p\in (1,\infty)$  there exists a positive constant $C$
  such that
 $$
 C^{-1}\,\|f\|_{\alpha,p}\le    \|S^{{\rm loc}}_{\alpha}f\|_p+\|f\|_p
 \le C \|f\|_{\alpha,p}\,; \smallskip
 $$
 \item [(ii)]  for any  $p>2d/(d+2\alpha)$ there exists a positive constant $C$
  such that
$$
 C^{-1}\,\|f\|_{\alpha,p}\le    \|D^{{\rm loc}}_{\alpha}f\|_p+\|f\|_p   \le C \|f\|_{\alpha,p}\,.
 $$
 \end{itemize}
 \end{theorem}  

We point out that the norm equivalence (i) of Theorem \ref{Salphaloc} is the
main tool that we use to 
prove Theorem \ref{t: productlemma2} in the case when $\alpha\in
(0,1)$, while the norm equivalence (ii) provides a further
characterization of the Sobolev norm for certain values of $p$ and
$\alpha$. 
 In order to prove Theorem \ref{t: productlemma2} in the case $\alpha\in
[1,\infty)$,  
we need to prove the $L^p$-boundedness of the local Riesz
transforms.  
We consider the collection of multiindices 
$$
\{1,\dots,q\}^m = \big\{ J=(j_1,\dots,j_m):\, j_\ell\in
\{1,\dots,q\},\, \text{for\ }\ell=1,\dots,m \big\} \,.
$$
For every $c>0$ and $J\in\{1,\dots,q\}^m$, we shall denote by $R^{c}_J$
the local Riesz transform of order $m$
\begin{equation}\label{local-Riesz-trans-def}
R^c_J=X_J(cI+\Delta)^{-m/2} \,, 
\end{equation}
where, 
$X_J=X_{j_1} \dots X_{j_m}$.
 
Then, we prove the following
boundedness result for $R^c_J$, whose statement involves a Hardy type
space $\mathfrak{h}^1(\rho)$ and a space $\mathfrak{bmo}(\rho)$, whose
precise definition is given in Subsection \ref{h1-bmo-subsec}. 
\begin{theorem}\label{local-Riesz-trans-Lp}
Let $G$ be a noncompact connected Lie group. There 
 exists $c>0$ sufficiently large such that for every $J\in
 \{1,\dots,q\}^m$ and $m\in\mathbb N$, 
the local Riesz transform $R^c_J$ is bounded from the Hardy space
$\mathfrak{h}^1(\rho)$ to $L^1(\rho)$, from $L^{\infty}$ to
$\mathfrak{bmo}(\rho)$ and on $L^p(\rho)$ for every $p\in (1,\infty)$. 
\end{theorem}
\medskip

Given a (sub)Laplacian $\Delta$ 
on a Lie group the question of the $L^p$-boundedness of the Riesz transforms
$\cR_j=X_j \Delta^{-1/2}$, and of their higher order analogue  
$\cR_J=X_J \Delta^{-m/2}$, 
where $J\in \{1,\dots,q\}^m$,  
has also a long and rich hystory.   
It is well known that the $L^p$-boundedness of the Riesz transforms
$\cR_j$, $j=1,\dots,q$,  is tighly connected to the equivalence of two  
natural definitions of homogeneous first order $L^p$  Sobolev spaces. 
The Riesz transforms
$\cR_j$ are known to be bounded on $L^p$ when 
the underlying group is stratified  \cite{F}, nilpotent \cite{LV}, 
of polynomial growth \cite{Ale} and on certain classes of Lie groups
of exponential growth \cite{HS, LM, Sj, SV}.  On nilpotent Lie groups
the Riesz transforms of higher order $\mathcal R_J$ are also bounded on $L^p$
\cite{ERS}, but it is known that this is not always the case (see
\cite{GQS} for an example of nonunimodular Lie group of exponential
growth where the Riesz transforms $\mathcal R_J$  
of order $2$   are unbounded on
$L^p$ for every $p\in [1,\infty)$).

In this paper we deal only with the local Riesz transforms.  The $L^p$
boundedness of the local Riesz transforms Theorem
\ref{local-Riesz-trans-Lp} 
is known to hold on nonamenable Lie groups \cite{L} and on every Lie
group when $\Delta$ is a complete Laplacian \cite{R}. 
Thus, it is
certainly an expected result, and maybe considered ``folklore'' by
many.  However,  to the best of our knowledge  this result is new in
the general setting of any subLaplacian on any noncompact Lie group, especially for the
 endpoint results. %,  
%and in fact we believe that its proof is nontrivial. 
 
 \smallskip
 
We point out   that the problems considered in this paper, namely the
algebra property of Sobolev spaces and the $L^p$ boundedness of local
and global Riesz transforms, have been intensively studied also in the
context of Riemannian manifolds.  Without any pretense of 
exhaustiveness, 
we refer the reader to \cite{AC, ACDH, B, CD, MMV, Ru,
  Strichartz} and the references therein for the study of the
boundedness of Riesz transforms and to \cite{BBR, BCF, CRTN} for
Sobolev algebras on Riemannian manifolds satisfying suitable geometric
assumptions.

Finally, we mention that the Sobolev algebra property is of great
importance in the study of the well-posedness of Cauchy problems
involving the operator $\Delta$ in some nonlinear differential
equation, such as a nonlinear heat equation, or a nonlinear
Schr\"odinger equation, see \cite{BBR,Tao,PV}. 
\medskip

The paper is organized as follows. In Section \ref{preliminaries} we
recall all preliminaries and notation on nonunimodular  Lie groups,
the properties of the maximal functions, the estimates of the heat
kernel associated with $\Delta$ and the definition of the Hardy and
BMO spaces that will be used in the paper. Section \ref{secriesz} is
devoted to the study of the boundedness of local Riesz transforms of
any order associated with $\Delta$, and we also prove that the
 analogue  of Strichartz and Bohnke \cite{Strichartz-multi} and
\cite{B} results cited earlier cannot hold in a generic nonunimodular Lie group.   
In Section \ref{secalpha01} we
prove two representation formulas for the Sobolev norms in the case
when $\alpha$ is in $(0,1)$. Section \ref{secproof} is devoted to the
proof of Theorem \ref{t: productlemma2}, while we collect in Section
\ref{secfinal} some final comments and a discussion on the future
developments of this work.  

\medskip

Given two non-negative 
quantities $A$ and $B$, we write $A \lesssim B$ to indicate
that there is $C>0$ such that $A\le CB$, and the constant $C$ does not
depend on the relevant parameters involved in $A$ and $B$.  We also
write $A\approx B$ when $A\lesssim B$ and $B\lesssim A$. 
\medskip

We wish to thank the anonimous referee for her/his careful reading of
the manuscript and for making several useful comments.

\section{Preliminaries}\label{preliminaries}

The Carnot--Carath\'eodory metric on $G$ associated with ${\bf{X}}$ is
defined as follows. An absolutely continuous curve
$\gamma:[0,1]\rightarrow G$ is called horizontal if
$\gamma'(t)=\sum_{j=1}^qa_jX_j(\gamma(t))$ for every $t\in [0,1]$. The
length of such a curve is defined as
$\ell(\gamma)=\int_0^1\big(\sum_{j=1}^q|a_j|^2\big)^{1/2}\di t$. The
distance of two points $x,y\in G$ is defined as the infimum of the
lengths of all horizontal curves joining $x$ to $y$ and denoted by
$d_C(x,y)$. 
 Since the vector
fields $\{X_j\}_{j=1}^q$ are left invariant, the metric $d_C$ is left
invariant.  We denote by $|x|$ the distance of a point $x\in G$ from
the identity $e$ 
of $G$ in such metric. For every $x_0\in G$ and $r>0$, the open ball
centred at $x_0$ of radius 
$r$ is $B(x_0,r)=\{x\in G:d_C(x,x_0)<r\}$.  
When $x_0=e$, 
we simply write $B_r=B(e,r)$, and set
\begin{equation}\label{Vr}
V({r})=\rho(B_r)=\lambda(B_r)\,.
\end{equation} 
Notice that for every $x_0\in G$ and $r>0$
\begin{equation}\label{pallainx0}
\rho(B(x_0,r))=\delta^{-1}(x_0)\,V(r)\,.
\end{equation}
It is known \cite{G, V1, V2} that there exists a positive constant $d$ such that
\begin{equation}\label{pallepiccole}
V({r})\approx r^d\qquad \forall r\in (0,1]\,,
\end{equation}
and there exists $D>0$ such that 
\begin{equation}\label{pallegrandi}
V({r})\lesssim e^{Dr}\qquad \forall r\in (1,\infty)\,.
\end{equation}
Notice in particular that the space $(G,d_C,\rho)$ is locally
doubling. Moreover, there exists a constant $Q>0$ such that 
\begin{equation}\label{intdelta12}
\int_{B_r}\delta^{1/2}\di\rho\lesssim r^{d} \,e^{Q r}\qquad \forall r>0\,. 
\end{equation}
Indeed, when $r\in (0,1]$  
$$
\int_{B_r}\delta^{1/2}\di\rho\lesssim
\|\delta^{1/2}\|_{L^{\infty}(B_1)}\,V({r})\lesssim r^d\,; 
$$
when $r\in (1,\infty)$  
$$
\int_{B_r}\delta^{1/2}\di\rho\lesssim (\sup_{B_r}
\delta)^{1/2}\,V({r})\lesssim e^{Ar }\,e^{ Dr}\,, 
$$
where $D$ is the constant in \eqref{pallegrandi} and
$A=\frac{1}{2}\big(\sum_{i=1}^q |(X_i\delta)(e)|^2\big)^{1/2}$ (see
\cite[Proposition 5.7 (ii)]{HMM}) that gives \eqref{intdelta12}. Let
us mention that the fact that the integral of $\delta^{1/2}$ on any
ball grows at most exponentially with respect to the radius of the
ball is crucial in the proof of our results.  

\bigskip

In the sequel we shall often deal with left invariant operators on $G$
and their kernels. Recall that by the Schwartz kernel theorem, all
bounded operators $T : C^\infty_c(G) \to \mathcal D'(G)$ have an
integral kernel $K_T \in \mathcal D'(G \times G)$, such that 
\[
Tf(x) = \int_G K_T(x,y) \, f(y) \di \rho(y)
\]
in the sense of distributions. Further, if $T$ is left invariant, then
it admits a convolution kernel $k_T \in \mathcal D'(G)$, such that 
\begin{equation}
T f(x) = f*k_T(x) = \int_G f(xy^{-1}) \,k_T(y) \di\rho(y);
\end{equation}
in this case the convolution kernel $k_T$ is related to the integral kernel $K_T$ by
\begin{equation}\label{eq:conv_int_kernel}
K_T(x,y) = k_T(y^{-1} x) \, \delta(y).
\end{equation}

We shall list below some notation and well-known results which will be used in the sequel. 

\subsection{Heat kernel estimates}\label{pt}
Let $p_t$ be the heat kernel of $\Delta$ at time $t$, i.e. the
convolution kernel $k_{T_t}$ of the operator $T_t=e^{-t\Delta}$ and
let $P_t$ be the corresponding integral kernel. By
\cite[Section IX]{VCS} there exist positive constants $c_1 \dots, c_4$
such that for every $x\in G$ and $t\in (0,1)$:  
\begin{itemize}
\item[(i)] $\int_Gp_t\di\rho=1$; \smallskip
\item[(ii)] $\delta^{1/2}(x)\, V(\sqrt t)^{-1}e^{-c_1|x|^2/t} \lesssim
  p_t (x)\lesssim  \delta^{1/2}(x)\, V(\sqrt t)^{-1}e^{-c_2|x|^2/t}
  $;  \smallskip
\item[(iii)] $|\frac{\partial p_t}{\partial t}(x)|\lesssim
  \delta^{1/2}(x)\,t^{-1}\,V(\sqrt t)^{-1}e^{-c_3|x|^2/t} $;  \smallskip
\item[(iv)] $|X_i p_t (x)|\lesssim \delta^{1/2}(x)\,t^{-1/2}\,V(\sqrt 
  t)^{-1}e^{-c_4|x|^2/t}   $\,. 
\end{itemize}
By \cite[p. 132]{tER} there exist $\omega\geq 0, b>0$ such that for
every multiindex  
$J\in\{1,\dots,q\}^m$   and every $t>0$ 
\begin{equation}\label{stimeXJpt}  
|X_J p_t(x)|  
\lesssim t^{-\frac{d+m}{2}}\,e^{\omega t}\,e^{-b|x|^2/t}\qquad \forall x\in G\,.
\end{equation}
By using the heat semigroup $T_t$ for every $\beta>0$ we define the 
 $g$-function 
$$
g_{\beta}f=\Big( \int_0^{\infty} |(t\Delta)^{\beta}T_tf|^2\frac{\di
  t}{t} \Big)^{1/2}\,. 
$$

Since $T_t$ is a diffusion semigroup symmetric with
respect to the 
measure $\rho$, it is well known that for every $p\in (1,\infty)$ and
every $f\in L^p$ 
\begin{equation}\label{galphap}
\|f\|_{p}\approx \|g_{\beta}f\|_p\,.
\end{equation}
See \cite{M, S}. 

\subsection{Maximal functions}
For every $R>0$ we define $\mathcal B_R$ as the set of all balls of
radius $\leq R$. The corresponding local Hardy--Littlewood maximal
function with respect to the right Haar measure is given by  
\begin{equation}\label{MR}
M^{R}f(x)= \sup_{B\in\mathcal B_R, x\in B} \frac{1}{\rho(B)} \int_B|f|\di\rho\,.
\end{equation}
The operator $M^{R}$ is bounded on $L^p$ for every $p\in (1,\infty]$ and it is of
weak type $(1,1)$. We also introduce the modified local Hardy--Littlewood maximal
function, with parameter $\beta\in [0,1)$ with respect to the right
Haar measure, given by  
\begin{equation}\label{MRbeta}
M^R_{\beta}f(x)= \sup_{B\in\mathcal B_R, x\in B} \frac{1}{\rho(B)^{1-\beta}} \int_B|f|\di\rho\,.
\end{equation}
It is easy to show that $M^R_{\beta}$ is bounded from
$L^{\frac{1}{\beta}}$ to $L^{\infty}$ and from $L^1$ to the Lorentz
space $L^{\frac{1}{1-\beta},\infty}$, so that by interpolation it is
bounded from $L^p$ to $L^q$ whenever $\frac1q=\frac1p-\beta$ and $p\in
(1,\frac{1}{\beta}]$.

We denote by $
\mathcal M_0$ the local heat maximal function defined by 
$$
\mathcal M_0 f=\sup_{0<t\leq 1}|f\ast p_t|\,.
$$
It is known that $
\mathcal M_0$ is bounded on $L^p$ for $p\in (1,\infty)$ \cite{CGGM, S}.

It is easy to see that the statements of \cite[Propositions
7-8-9-10]{CRTN} which concern global maximal operators can be
reformulated for the local maximal functions $\mathcal M_0$ and $M^R$:
indeed only the $L^p$-boundedness for $p\in (1,\infty)$ of the local maximal
functions and the local doubling property are neeeded to adapt the
proofs of \cite[Propositions 7-8-9-10]{CRTN} to our setting.

\subsection{The spaces $\mathfrak{h}^1(\rho)$ and 
$\mathfrak{bmo}(\rho)$}\label{h1-bmo-subsec}

The theory of Hardy spaces of Goldberg type developed in \cite{MeVo}
applies to the space $(G,d_C,\rho)$. For the reader's convenience, we
recall here briefly the definition of the atomic Hardy space
$\mathfrak h^1(\rho)$ and its dual $\mathfrak{bmo}(\rho)$ and a few
related results. We refer the reader to \cite{Go} for details on the
theory of Goldberg Hardy spaces in the Euclidean setting and to
\cite{MeVo, T} for the corresponding theory in the context of metric
spaces and Riemannian manifolds. 
 
\begin{definition}
A standard atom at scale $1$ is a function $a\in L^1$ supported in a
ball $B\in \mathcal B_1$ such that 
\begin{itemize} 
\item[(i)] $\|a\|_{2}\leq \rho(B)^{-1/2}$;
\item[(ii)] $\int a\di\rho=0$.
\end{itemize}
A global atom at scale $1$ is a function $a\in L^1$ supported in
a ball $B$ of radius exactly $1$ such that $\|a\|_{2}\leq
\rho(B)^{-1/2}$.  
Standard and global atoms at scale $1$ will be referred to as atoms at
scale $1$. %We denote by $\Atoms(\mu_X)$ the collection of all atoms
           %at scale $1$. 
\end{definition}
The Hardy space $\mathfrak h^1(\rho)$ is defined as the space 
$$
\mathfrak h^1(\rho)=\left\{f\in L^1(\rho): \,\, f=\sum_k c_k a_k,\,
  a_k\,\,{\rm{atom\,\, at\, scale}}\, 1,\, c_k \in \CC,\,
  \sum_k|c_k|<\infty \right\}, 
$$
endowed with the usual atomic norm
$$
\|f\|_{\mathfrak h^1}=\inf\left\{ \sum_k |c_k|:\,\, f = \sum_k c_k
  a_k,\, a_k\,\,{\rm{atom\,\, at\, scale}}\, 1,\, c_k \in \CC
\right\}. 
$$

By \cite[Theorem 2]{MeVo} the dual of $\mathfrak h^1(\rho)$ can be
identified with the space $\mathfrak{bmo}(\rho)$ of all equivalence
classes of locally integrable functions $g$ modulo constants such that 
$$
\begin{aligned}
\|g\|_{\mathfrak{bmo}} 
&:=\sup_{B\in\mathcal B_1} \left( \frac{1}{\rho(B)} \int_B |g-g_B|^2
  \di\rho \right)^{1/2}  \\
& \qquad +  \sup_{x\in G} \left( \frac{1}{\rho(B(x,1))} \int_{B(x,1)} |g|^2 \di\rho \right)^{1/2} < \infty\,,
\end{aligned}
$$
where $g_B = \rho(B)^{-1} \int_B g \di \rho$.

By \cite[Theorem 8.2]{CMM1} and \cite[Proposition 4.5]{CMM2} the
following criterion for the boundedness of integral operators on $G$
holds. 

\begin{proposition}\label{prp:hardy_hormander}
If $T$ is a  bounded operator on $L^2$ and its integral kernel
 $K_T$ is a locally integrable function away from the diagonal of $G\times G$ such that 
\begin{equation}\label{eq:hormander-cmm}
  \sup_{B\in \mathcal B_1} \sup_{y,z\in B} \int_{(2B)^c}|K_T(x,y)-K_T(x,z)| \di \rho(x) < \infty
\end{equation} 
and
\begin{equation}\label{eq:hormander-mv}
   \sup_{y\in G} \int_{(B(y,2))^c}|K_T(x,y)| \di \rho(x) < \infty ,
\end{equation} 
then $T$ is bounded from $\mathfrak h^1(\rho)$ to $L^1$.

If $T$ is a  bounded operator on $L^2$ and its integral kernel
 $K_T$ is a locally integrable function off the diagonal of $G\times G$ such that 
\begin{equation}\label{eq:hormander-cmmdual}
  \sup_{B\in \mathcal B_1} \sup_{y,z\in B} \int_{(2B)^c}|K_T(y,x)-K_T(z,x)| \di \rho(x) < \infty
\end{equation} 
and
\begin{equation}\label{eq:hormander-mvdual}
  \sup_{y\in G} \int_{(B(y,2))^c}|K_T(y,x)| \di \rho(x) < \infty ,
\end{equation} 
then $T$ is bounded from $L^{\infty}$ to $\mathfrak{bmo}(\rho)$. 
\end{proposition}

Furthermore, by  \cite[Theorem 5]{MeVo}, the following interpolation result holds, where $(V,W)_{[\theta]}$ denotes the lower complex interpolation space of parameter $\theta \in (0,1)$ between the Banach spaces $V,W$ (see \cite{BL}).

\begin{theorem}\label{thm:hardy_interpol}
Let $\theta \in (0,1)$ and set $p_\theta = 2/(2-\theta)$. Then
$(\mathfrak h^1(\rho),L^2)_{[\theta]} = L^{p_\theta}$ and
$(\mathfrak{bmo}(\rho),L^2)_{[\theta]} = L^{p'_\theta}$.  
\bigskip
\end{theorem}

\section{Boundedness of the local Riesz transforms}\label{secriesz}

Recall that for every multiindex $J\in\{1,\dots,q\}^m$ 
the local Riesz transform of order $m$ is defined by
$R^c_J=X_J(cI+\Delta)^{-m/2}$ (see \eqref{local-Riesz-trans-def}).

In order to prove the 
boundedness  of the local  Riesz transforms on 
$L^p$,  we shall need to apply the following result concerning
the derivatives of the heat kernel at small times.
The techniques used in the proof of the following lemma follow
closely those used in \cite{Gr, Ru}.  
\begin{lemma}\label{q_t}
Let $B=B(c_B,r_B)$ be a ball of radius $r_B\leq 1$, $t\in
[r_B^2,1]$ and $y,z\in B$. For every $x\in G$ define  
$$
q_t(x)=P_t(x,y)-P_t(x,z)\,. 
$$
Then there exist $\gamma\in (0,1)$, $c_5>0$  such that 
\begin{itemize}
\item[(i)] $|q_t(x)|\lesssim \delta^{1/2}(c_Bx)\,V(\sqrt
  t)^{-1}\,\Big(\frac{d(y,z)}{\sqrt
    t}\Big)^{\gamma}\,e^{-c_5|c_B^{-1}x|^2/t}$;\smallskip 
\item[(ii)] for every $\beta<2c_5$, 
$$
\int_{2r_B\leq |c_B^{-1}x|\leq 2}|q_t(x)|^2  e^{\beta
  |c_B^{-1}x|^2/t}\di\rho(x)\lesssim \delta(c_B)\,V(\sqrt
t)^{-1}\,\Big(\frac{d(y,z)}{\sqrt t}\Big)^{2\gamma}\,; 
$$
\item[(iii)] for every $\beta<2c_5$ and  $J\in\{1,\dots,q\}^m$, $m$
  a non-negative integer,
$$
\int_{2r_B\leq |c_B^{-1}x|\leq 2}|X_Jq_t(x)|^2  e^{\beta
  |c_B^{-1}x|^2/t}\di\rho(x)\lesssim t^{-m}\,  \delta(c_B)\,V(\sqrt
t)^{-1}\,\Big(\frac{d(y,z)}{\sqrt t}\Big)^{2\gamma}\,. $$   
\end{itemize}
\end{lemma} 

\proof %%\begin{proof}
We first prove (i). Fix $x\in G$ and define 
$u(t,\cdot)=P_t(x,\cdot)$. 
 Then $u$ is a solution of the heat equation
$(\partial_t+\Delta)u=0$. By applying \cite[Proposition 10]{Ru} (see
also \cite[Proposition 3.2]{SC}) we deduce that there exists
$\gamma\in (0,1)$ such that  
$$
\begin{aligned}
|q_t(x)|&\lesssim \Big(\frac{d(y,z)}{r_B}\Big)^{\gamma}\sup_{(\tau,w)\in Q} P_{\tau}(x,w)\\
&\lesssim \Big(\frac{d(y,z)}{\sqrt t}\Big)^{\gamma}\sup_{(\tau,w)\in
  Q} \delta (w)\delta^{1/2}(w^{-1}x)\,V(\sqrt
\tau)^{-1}\,e^{-c_2|w^{-1}x|^2/t}\,, 
\end{aligned}
$$
where $Q=(\frac{4}{9} t,\frac{20}{9} t)\times B\Big(y,\frac{4}{3}\sqrt
t\Big)$. Using the fact that $y,z\in B$, $r_B\leq 1$ and $r_B^2\leq
t$, we deduce that there exists $c_5>0$ such that 
$$
|q_t(x)|\lesssim \Big(\frac{d(y,z)}{\sqrt t}\Big)^{\gamma}
\delta^{1/2}(c_Bx)\,V(\sqrt t)^{-1}\,e^{-c_5|c_B^{-1}x|^2/t}  \qquad
\forall x\in G\,. 
$$
To prove (ii) we apply (i) and the change of variables $c_B^{-1}x=v$ to obtain 
$$
\begin{aligned}
& \int_{2r_B\leq |c_B^{-1}x|\leq 2}|q_t(x)|^2  e^{\beta
  |c_B^{-1}x|^2/t}\di\rho(x) \\
&\quad
\lesssim \Big(\frac{d(y,z)}{\sqrt t}\Big)^{2\gamma} V(\sqrt t)^{-2}
\int_{2r_B\leq |c_B^{-1}x|\leq 2}
\delta(c_Bx)\,e^{(-2c_5+\beta)|c_B^{-1}x|^2/t} \di\rho(x)\\ 
& \quad
\lesssim  \Big(\frac{d(y,z)}{\sqrt t}\Big)^{2\gamma} V(\sqrt t)^{-2}
\delta(c_B) \int_{2r_B\leq |v|\leq 2}\,e^{(-2c_5+\beta)|v|^2/t}
\di\lambda(v)\,.
\end{aligned}
$$
If $2r_B\leq \sqrt t$ we choose $j_0$ as the smallest integer such that $2^{j_0+1}\sqrt t \geq
2$ and obtain 
$$
\begin{aligned}
& \int_{2r_B\leq |c_B^{-1}x|\leq 2}|q_t(x)|^2  e^{\beta
  |c_B^{-1}x|^2/t}\di\rho(x) \\
& 
\lesssim  \Big(\frac{d(y,z)}{\sqrt t}\Big)^{2\gamma} V(\sqrt t)^{-2}  \delta(c_B) \\
&\qquad \times \Big[\int_{2r_B\leq |v|\leq \sqrt t } \di\lambda(v)
+\sum_{j=0}^{j_0} \int_{2^j\sqrt t\leq |v|\leq 2^{j+1}\sqrt t
}e^{(-2c_5+\beta)2^{2j}}  \di\lambda(v)  \Big]\\
& \lesssim  \Big(\frac{d(y,z)}{\sqrt
  t}\Big)^{2\gamma} V(\sqrt t)^{-2}  \delta(c_B)\Big[ V(\sqrt t)+ \sum_{j=0}^{j_0} e^{(-2c_5+\beta)2^{2j}}  (2^{j+1}\sqrt t)^d \Big]\\
& \lesssim  \Big(\frac{d(y,z)}{\sqrt
  t}\Big)^{2\gamma} V(\sqrt t)^{-1}  \delta(c_B)\,,
\end{aligned}
$$
where we used the fact that $\beta<2c_5$. The proof in the case when
$\sqrt t\leq 2r_B$ is similar and is omitted.  

We now prove (iii). Take $J\in\{1,\dots,q\}^m$, where $m$ is a non-negative integer. For every $x\in G$, the function $(t,y) \mapsto  X_{J,x}P_t(x,y)$ is a solution of the heat equation $\partial_t u +\Delta u=0$. Thus, by~\cite[Proposition 3.2]{SC} 
$$
|X_{J} q_t(x)| \lesssim \left( \frac{d_C(y,z)}{\sqrt{t}}\right)^{\gamma} \sup_{(\tau, w)\in Q} |X_{J,x} P_\tau (x,w)|
$$
for some $\gamma \in (0,1)$, where $Q=\left(\frac{4}{9} t,\frac{20}{9} t\right )\times B\left(y,\frac{4}{3}\sqrt t\right)$. By means of  \eqref{stimeXJpt} and the assumptions on $r_B, t, y, z$ we get
$$
\begin{aligned}
|X_{J,x} P_\tau (x,w)|
& \lesssim    \tau^{-(d+m)/2} \,e^{\omega \tau}\,e^{-b|w^{-1}x|^2/\tau}\,.
\end{aligned}
$$
Since $\tau \approx t$, $t\in (0,1)$ and $t^{-d/2}\approx V(\sqrt{t})^{-1}$, there exists a constant $c_5$ such that 
\begin{align}\label{XJqPuntuale}
|X_{J} q_t(x)|  \lesssim \left( \frac{d_C(y,z)}{\sqrt{t}}\right)^{\gamma}   t^{-m/2}V(\sqrt{t})^{-1} e^{-c_5 |c_B^{-1}x|^2/t}\,.
\end{align}
If $\beta<2c_5$ we can apply \eqref{XJqPuntuale}, use the change of variables $c_B^{-1}x=v$ and argue as in the proof of (ii) to obtain that
$$
\int_{2r_B\leq |c_B^{-1}x|\leq 2}|X_Jq_t(x)|^2  e^{\beta
  |c_B^{-1}x|^2/t}\di\rho(x)\lesssim t^{-m}\,  \delta(c_B)\,V(\sqrt
t)^{-1}\,\Big(\frac{d(y,z)}{\sqrt t}\Big)^{2\gamma}\,. 
\qed  \medskip
$$    
 
%% \end{proof}

We are now ready to prove Theorem \ref{local-Riesz-trans-Lp}.

\proof[Proof of Theorem \ref{local-Riesz-trans-Lp}] Fix $J\in\{1,\dots,q\}^m$. By~\cite[Theorem 4.8, IV]{TER1} if $c>0$ is large enough then the local Riesz transforms $R_J^c$ is bounded on $L^2$. The convolution
kernel of $R^c_J$ is given by 
$$
k^c_J(x)=c_J\int_0^1 t^{m/2-1} e^{-ct}X_Jp_t(x)\di
t+c_J\int_1^{\infty} t^{m/2-1} e^{-ct}X_J  p_t(x)\di
t=k^0(x)+k^{\infty}(x)\,. 
$$  
We claim that $k^{\infty}\in L^1$ if $c$ is sufficiently
large. Indeed, we can choose $c$ sufficiently large so that 
$$
|k^{\infty}(x)|\lesssim \int_1^{\infty} t^{m/2-1} e^{-ct}
t^{-(d+m)/2} e^{\omega t}\,e^{-b|x|^2/t}  \di t\lesssim
\int_1^{\infty}  e^{-c't}   \,e^{-b|x|^2/t}  \di t\,, 
$$
for some $c'>0$. Now define for every $t>0$, $A^i_t=B_{2^{i}\sqrt
t}\setminus B_{2^{i-1}\sqrt t}$, $i\geq 1$. Then by \eqref{pallegrandi}
$$
\begin{aligned}
\int_G|k^{\infty}(x)|\di\rho(x)
&\lesssim \int_1^{\infty}
\int_{B(e\sqrt t)}  e^{-c't}   \,e^{-b|x|^2/t}    \di\rho(x)\di
t+\sum_{i=1}^{\infty} \int_1^{\infty} \int_{A^i_t  }  e^{-c't}
\,e^{-b|x|^2/t} \di\rho(x)\di t\\ 
&\lesssim \int_1^{\infty}  e^{-c't}\ e^{D \sqrt t}\di
t+\sum_{i=1}^{\infty} \int_1^{\infty} e^{-c't} e^{-b 2^{2j}}
e^{D2^i\sqrt t} \di t\\ 
&\lesssim 1+ \sum_{i=1}^{\infty} e^{-b
  2^{2i}+\frac{D^22^{2i}}{4c'}}\int_1^{\infty} e^{-( \sqrt{c'
    t}-\frac{D2^i}{2\sqrt c'} )^2   }\di t\\ 
&\lesssim 1\,,
\end{aligned}
$$ 
if $c$ is sufficiently large. %Notice that we can show in the same way
                              %that $k^{\infty}\in L^1(\lambda)$ if
                              %$c$ is sufficiently large. 
Thus the convolution operator $f\mapsto f\ast k^{\infty}$ is bounded
on $L^p$ for every $p\in [1,\infty]$, and a fortiori it is bounded
from $\mathfrak h^1(\rho)$ to $L^1$ and from $L^{\infty}$ to
$\mathfrak{bmo}(\rho)$.  

We now consider the kernel $k^0$. Choose a smooth cutoff function
$\phi$ supported in $B_1$ such that $0\leq \phi\leq 1$. First notice
that   
$$
|(1-\phi(x))k^0(x)|\lesssim
\int_0^1t^{m/2-1}t^{-d/2-m/2}e^{-b|x|^2/t}\di
t=\int_{|x|^2}^{\infty} \Big(\frac{|x|^2}{u}\Big)^{-d/2}e^{-bu}\di
u\lesssim e^{-b' |x|^2}\,, 
$$
for some $b'>0$. Then
\begin{equation}\label{intk0}
\begin{aligned}
\int_G|(1-\phi(x))k^{0}(x)|\di\rho(x)&\lesssim  \sum_{i=1}^{\infty}
  \int_{A^i_1  }  e^{-b'|x|}     \di\rho(x) \\ 
&\lesssim \sum_{i=1}^{\infty}   e^{-b' 2^{2i}} e^{D2^i} \\
&\lesssim 1\,.
\end{aligned}
\end{equation}
Thus the convolution operator $f\mapsto f\ast (1-\phi)k^0$ is  bounded
on $L^p$ for every $p\in [1,\infty]$, and a fortiori it is bounded
from $\mathfrak h^1(\rho)$ to $L^1$ and from $L^{\infty}$ to
$\mathfrak{bmo}(\rho)$.  

 It remains to consider the operator $f\mapsto f\ast (\phi k^0)$ which
 is bounded on $L^2$, as difference of operators bounded on
 $L^2$. Let us denote by $\ell$ the function $\phi k^0$ and by
 $L$ the integral kernel corresponding to the convolution operator
 with kernel $\ell$, i.e. $L(x,y)=\ell(y^{-1}x)\delta(y)$. Notice that
$$
|\ell(x)|\lesssim \int_0^1t^{m/2-1}t^{-d/2-m/2}e^{-b|x|^2/t}\di
t=\int_{|x|^2}^{\infty} \Big(\frac{|x|^2}{u}\Big)^{-d/2}e^{-bu}\di
u\lesssim |x|^{-d}\,, 
$$
and, for every $j=1,\dots,q$,
\begin{equation}\label{Xjl}
\begin{aligned}
|X_j\ell(x)|&\lesssim |k^0(x)|+\int_0^1t^{m/2-1}t^{-d/2-m/2-1/2}e^{-b|x|^2/t}\di t\\
& \lesssim |x|^{-d}+|x|^{-d-1}\\
&\lesssim |x|^{-d-1}\,.
\end{aligned}
\end{equation}
Notice that since $\ell$ is supported in $B_1$ for every ball $B$ of radius $1$ and every $y\in B$
\begin{equation}\label{inthcpalle1}
\int_{(2B)^c}|L(y,x)|\di\rho(x)=\int_{(2B)^c}|L(x,y)|\di\rho(x)=0\,.
\end{equation}
Take now a ball $B=B(c_B,r_B)$ of radius $r_B\leq 1$. For every $y,z\in B$ we have 
$$
\begin{aligned}
\int_{(2B)^c}|L(y,x)-L(z,x)|\di\rho(x)&=\int_{(2B)^c}| \ell(x^{-1}y)\delta(x)  -   \ell(x^{-1}z)\delta(x) |\di\rho(x)\\
&=\int_{ \{x:d(x^{-1},c_B)>2r_B }| \ell(xy)  -   \ell(xz)  |   \di\rho(x)\\
&=\int_{|u|>2r_B}|\ell(uc_B^{-1}y)-\ell(uc_B^{-1}z)|\di\rho(u)\\
&=\int_{2r_B<|u|\leq 2}|\ell(uc_B^{-1}y)-\ell(uc_B^{-1}z)|\di\rho(u)\\ 
 &\lesssim d(c_B^{-1}y,c_B^{-1}z)\int_{2r_B\leq |u|\leq 2} \sum_{j=1}^q|X_j\ell(u)|\di\rho(u)\\
 &\lesssim r_B\,\int_{2r_B\leq |u|\leq 2} |u|^{-d-1}\di\rho(u)\,,
 \end{aligned}
$$ 
where we have applied \eqref{Xjl}. Choose $i_0$ as the biggest integer
such that $2^{i_0-1}\leq 2r_B$ and consider the  annuli
$A^i=B_{2^i}\setminus B_{2^{i-1}}$, with $i_0\leq i\leq 1$. By \eqref{pallepiccole} we
obtain  
\begin{equation}\label{inthcpallepiccole}
\begin{aligned}
\int_{(2B)^c}|L(y,x)-L(z,x)|\di\rho(x)&\lesssim r_B\,\sum_{i=i_0}^{1} \int_{A^i} |u|^{-d-1}\di\rho(u)\\
&\lesssim r_B\, \sum_{i=i_0}^{1}2^{-i(d+1)} 2^{id}\\
&\lesssim r_B\,2^{-i_0}\\
&\lesssim 1\,.
\end{aligned}
\end{equation}
By Propostion \ref{prp:hardy_hormander} the conditions \eqref{inthcpalle1} and \eqref{inthcpallepiccole} imply that $R^c_J$ is bounded from $L^{\infty}$ to
$\mathfrak{bmo}(\rho)$.  By Theorem \ref{thm:hardy_interpol} we
deduce that $R^c_J$ is bounded on $L^p$ for every $p\in
[2,\infty)$.  

We now prove that $f\mapsto f\ast k^0$ is bounded from
$\mathfrak h^1(\rho)$ to $L^1$. We denote by $K^0(x,y)$ the
corresponding integral kernel. For every ball $B$ of radius $1$ and every $y\in B$
by \eqref{intk0} we have  
\begin{equation}\label{inthcpalle1bis}
\int_{(2B)^c}|K^0(x,y)|\di\rho(x)=\int_{d(c_B,x)\geq
  2}\delta(y)\,|k^0(y^{-1}x)|\di\rho(x)\leq \int_{|v|\geq 1 }
|k^0(v)|\di\rho(v)\lesssim 1 \,. 
\end{equation}
Take now a ball $B=B(c_B,r_B)$ of radius $r_B\leq 1$. For every $y,z\in B$ we have 
$$
\begin{aligned}
\int_{(2B)^c}|K^0(x,y)-K^0(x,z)|\di\rho(x)&=\int_{2r_B\leq d(x,c_B)\leq 2}|K^0(x,y)-K^0(x,z)|\di\rho(x)\\
&\qquad +\int_{ d(x,c_B)>2} |K^0(x,y)-K^0(x,z)|\di\rho(x)\\
&=I+I\!I\,.
\end{aligned}
$$
By applying \eqref{intk0} we have 
$$
\begin{aligned}
I\!I&\leq \int_{ d(x,c_B)>2} |K^0(x,y)|\di\rho(x)+\int_{ d(x,c_B)>2} |K^0(x,z)|\di\rho(x)\\
&\leq  \delta(y)\int_{d(v,y^{-1}c_B)>
  2}|k^0(v)|\delta^{-1}(y)\di\rho(v)+\delta(z)\int_{d(v,z^{-1}c_B)>
  2}|k^0(v)|\delta^{-1}(z)\di\rho(v)\\ 
&\leq  2\int_{|v|>1}|k^0(v)|\di\rho(v)\\
&\lesssim 1\,.
\end{aligned}
$$
To estimate the integral $I$ we first decompose it as follows: 
$$
\begin{aligned}
I&\leq \int_0^1t^{m/2-1}e^{-ct}\int_{2r_B\leq d(x,c_B)\leq 2}|
X_{J,x}P_t(x,y)-X_{J,x}P_t(x,z)    |\di\rho(x)\\ 
&=\int_0^{r_B^2}t^{m/2-1}e^{-ct}\int_{2r_B\leq d(x,c_B)\leq 2}|
X_{J,x}P_t(x,y)-X_{J,x}P_t(x,z)    |\di\rho(x)\\ 
&\qquad +\int_{r_B^2}^1t^{m/2-1}e^{-ct}\int_{2r_B\leq d(x,c_B)\leq
  2}| X_{J,x}P_t(x,y)-X_{J,x}P_t(x,z)    |\di\rho(x)\\ 
&=I_1+I_2\,.
\end{aligned}
$$
Since $P_t(x,y)=\delta(y)p_t(y^{-1}x)$, we have
$X_{J,x}P_t(x,y)=\delta(y) (X^Jp_t)(y^{-1}x)$, 
so that by \eqref{stimeXJpt}
$$
|X_{J,x}P_t(x,y)|\lesssim \delta(y)\,t^{-d/2-m/2}e^{\omega t}e^{-b|y^{-1}x|^2/t}\,. 
$$ 
We then have
$$
\begin{aligned}
I_1
&\lesssim \int_0^{r_B^2} t^{-1-d/2} \int_{2r_B\leq |c_B^{-1}x|\leq 2} \delta(y)e^{-b|y^{-1}x|^2/t}\di\rho(x)\\
&\lesssim \int_0^{r_B^2} t^{-1-d/2}\delta(c_B) \int_{2r_B\leq |v|\leq 2}e^{-b|v|^2/t}  \delta^{-1}(c_Bv)  \di\lambda(v)\\
&\lesssim  \int_{2r_B\leq |v|\leq 2} \int_{|v|^2/r_B^2}^{\infty}e^{-bs} \bigg(\frac{s}{|v|^2}\bigg)^{1+d/2} \frac{|v|^2}{s^2}\di s\di\rho(v)\\
&\lesssim  \int_{2r_B\leq |v|\leq 2} |v|^{-d}  \int_{|v|^2/r_B^2}^{\infty}e^{-b's}\di s\di\rho(v)\\
&\lesssim \int_{2r_B\leq |v|\leq 2} |v|^{-d}   e^{-b'\frac{|v|^2}{r_B^2}} \di\rho(v)\,.
\end{aligned}
$$
To estimate the last integral we split the domain of integration as
the union $\cup_{j=j_0}^1 A_j$, where $A_j=B_{2^j}\setminus
B_{2^{j-1}}$, where $j_0$ is the  largest  integer such that
$2^{j_0-1}\leq 2r_B$ and obtain 
$$
\int_{2r_B\leq |v|\leq 2} |v|^{-d}   e^{-b'\frac{|v|^2}{r_B^2}}
\di\rho(v)
 =\sum_{j=j_0}^1 2^{-dj} \,  e^{-b'\frac{2^{2j}}{r_B^2}  }
\,2^{dj}\lesssim 1\,.
$$
It remains to estimate the integral 
$$I_2=\int_{r_B^2}^1 t^{m/2-1}e^{-ct}\int_{2r_B\leq |c_B^{-1}x|\leq 2}|X_Jq_t(x)|\di\rho(x)\di t\,,
$$
where $q_t(x)=P_t(x,y)-P_t(x,z)$. By applying the Cauchy--Schwarz inequality 
$$
\begin{aligned}
\int_{2r_B\leq |c_B^{-1}x|\leq 2}|X_Jq_t(x)|\di\rho(x)&\leq
\Big(\int_{2r_B\leq |c_B^{-1}x|\leq 2}
|X_Jq_t(x)|^2e^{\beta|c_B^{-1}x|^2/t}\di\rho(x)  \Big)^{1/2}\\ 
&\qquad
\times \Big(  \int_{2r_B\leq |c_B^{-1}x|\leq 2}
e^{-\beta|c_B^{-1}x|^2/t}\di\rho(x)  \Big)^{1/2}\\ 
&=A_t\cdot B_t\,,
\end{aligned}
$$
with a constant $\beta$ such that $0<\beta<2c_5$, where $c_5$ is the constant which appears in Lemma \ref{q_t}. 

To estimate $B_t$, when $2r_B\leq \sqrt t$ we choose $j_0$ as the
smallest integer such that $2^{j_0+1}\sqrt t\geq 2$ and write   
$$
\begin{aligned}
(B_t)^2&\leq \delta^{-1}(c_B)\int_{2r_B\leq |v|\leq 2}e^{-\beta |v|^2/t}\di\lambda\\
&=\delta^{-1}(c_B)\Big(  \int_{2r_B\leq |v|\leq \sqrt t}e^{-\beta
  |v|^2/t}\di\lambda
+\sum_{j=1}^{j_0}  \int_{2^j\sqrt t\leq |v|\leq 2^{j+1}\sqrt t}e^{-\beta |v|^2/t}\di\lambda\Big)\\
&\lesssim  \delta^{-1}(c_B)\Big(  V(\sqrt t)e^{-\beta r_B^2/t}+\sum_{j=1}^{j_0}  (2^j\sqrt t)^de^{-\beta 2^j}\Big)\\
&\lesssim  \delta^{-1}(c_B)   V(\sqrt t)e^{-\beta r_B^2/t}\,.
\end{aligned}
$$
When $\sqrt t\leq 2r_B$ we argue in a similar way and obtain
$(B_t)^2\lesssim \delta^{-1}(c_B)   V(\sqrt t)e^{-\beta r_B^2/t}$. By
the previous estimate and Lemma \ref{q_t} we get 
\begin{multline*}
\int_{2r_B\leq |c_B^{-1}x|\leq 2}|X_Jq_t(x)|\di\rho(x) \\
\lesssim 
\delta(c_B)^{1/2}t^{-m/2}\,V(\sqrt t)^{-1/2}\delta^{-1/2}(c_B)
V(\sqrt t)^{1/2}e^{-\beta r_B^2/2t}\,\Big(\frac{r_B}{\sqrt t} \Big)^{\gamma}\,,
\end{multline*}
so that
$$
I_2\lesssim \int_{r_B^2}^1 t^{m/2-1}e^{-ct}  t^{-m/2}\,e^{-\beta
  r_B^2/2t}\Big(\frac{r_B}{\sqrt t} \Big)^{\gamma} \di t\lesssim 1\,.
$$
This shows that for every ball $B$ of radius $r_B\leq 1$ and every points $y,z\in B$ 
\begin{equation}\label{inthcpallepiccolebis}
\int_{(2B)^c}|K^0(x,y)-K^0(x,z)|\di\rho(x)\lesssim 1\,,
\end{equation}
which together with condition \eqref{inthcpalle1bis} implies that the integral
operator with integral kernel $K^0$ is bounded from $\mathfrak h^1(\rho)$ to
$L^1$. By Theorem \ref{thm:hardy_interpol} we deduce that $R^c_J$ is bounded on $L^p$ for every $p\in (1,2]$. 
\qed
\medskip

As a corollary  of Theorem \ref{local-Riesz-trans-Lp} we have the
following result involving Sobolev spaces.
\begin{corollary}\label{sobolevspaces}
For every $p\in (1,\infty)$ the following hold:
\begin{itemize}
\item[(i)] if $\ell\in\mathbb N$, then $f\in L^p_{\ell}$ if and only
     if $X_Jf\in L^p$ for every $J\in \{1,\dots,q\}^m$ with $m\leq \ell$
  and  
$$
\|f\|_{\ell,p}\approx\sum_{J\in \{1,\dots,q\}^m,\, m\le\ell} \|X_Jf\|_{p}\,.
$$  
\item[(ii)] for every $\alpha\geq 0$, $f\in L^p_{\alpha+1}$ if and only if $f\in L^p_{\alpha}$
  and $X_if\in L^p_{\alpha}$, $i=1,\dots,q$, and 
$$
\|f\|_{\alpha+1,p}\approx \|f\|_{\alpha,p}+\sum_{i=1}^q\|X_if\|_{\alpha,p}\,.
$$
\item[(iii)] for every 
$p\in (1,\infty)$, $\alpha\geq 0$ and $c>0$, 
$$
\|f\|_{\alpha,p}\approx \|(I+\Delta)^{\alpha/2}f\|_p\approx
\|(cI+\Delta)^{\alpha/2}f\|_p \qquad \forall f\in L^p_{\alpha}\,. 
$$
\end{itemize}
\end{corollary}
\begin{proof}
Statement (i) follows from the $L^p$-boundedness of local Riesz
transforms $R_J^c$ arguing as in \cite[Theorem
5.14]{R}. Statement (ii) has the same proof as \cite[Proposition
19]{CRTN}. 

To prove (iii) we observe that by \cite[Theorem 6.4]{K1} 
$L^p_\alpha=\operatorname{dom} \big( (I+\Delta)^{\alpha/2}\big)$, and then arguing
as in
 \cite[Propositions 3.16, 4.1]{F}, the equivalence of norms on the
 left
follows.  The one on the right now follows at once.
\end{proof}

We are now in the position to show that on a general nonunimodular
group $G$, the property that $L^p_\alpha(G)$ is an algebra when the
product $\alpha p$ is sufficiently large, cannot hold true.  We recall
that when $G$ is a nilpotent Lie group of homogeneuous dimension $Q$,
and $\Delta$ is a 
subLaplacian, then $L^p_\alpha(G)$ is an algebra provide $\alpha p>Q$,
see \cite{B} and the earlier paper \cite{Strichartz-multi} for the
case of $\R^n$.  The counterexample appears in the case of the  ``$ax+b$-group''.
Precisely, let
$G = \R \ltimes \R_+$, with product
given by $(x,a)(x',a')=(x+ax',aa')$.  Then, 
the right Haar measure is $\di \rho(a,x) = a^{-1} \di a \di x$, 
$\delta(x,a)=a^{-1}$ and
a basis for the
left invariant vector fields is $\{X_0,X_1\}$, where 
$X_0 = a\partial_a$ and $X_1= a\partial_x$.  Then, we have the
following result.
\begin{theorem}\label{contro-esempio}
Let $G=\R \ltimes \R_+$.
Then, for every $p\in(1,+\infty)$ and $k$ positive integer, the Sobolev
space $L^p_k(G)$ is not an algebra.
\end{theorem}

\begin{proof}
Let $\psi$ be a nonnegative function in $C^\infty_c (0,1)$ such that $\psi=1$ on $[1/4,3/4]$, and  $\chi$ a nonnegative function in $C^\infty_c
 (-1,1)$ such that $\chi=1$ on $[0,1/2]$.  For $\gamma,r>0$ 
define 
\begin{equation}\label{g-controesem-def}
g(x,a) = \psi(x/a^r) \chi(a) a^{-\gamma} \,.
\end{equation}
We claim that if $\gamma$ and $r$ satisfy the condition
\begin{equation}\label{cond-gamma-r}
\frac{r}{2p}  <\gamma < \frac{r}{p} + k_1(1-r) \,,\qquad \text{for\ } k_1=0,1,\dots,k,
\end{equation}
then 
 $g\in L^p_k$, but $g^2\not\in L^p$.  Note that condition  \eqref{cond-gamma-r} is
satisfied
by
any pair $r,\gamma$ with $0<r<1$ and $r/(2p)<\gamma<r/p$.

By Corollary \ref{sobolevspaces} (i), in order to show that $g\in
L^p_k$, since $\{X_0,X_1\}$ is a basis, we need to show that
$X_0^{k_0}X_1^{k_1}g \in L^p$ when
$0\le k_0+k_1\le k$. Using induction, it is easy to check that 
\begin{itemize}
\item[{\tiny$\bullet$}] $X_1^{k_1}\psi(x/a^r )= a^{k_1-rk_1}\widetilde\psi(x/a^r)$ for
another   $\widetilde\psi\in C^\infty_c  (0,1)$; \smallskip
\item[{\tiny$\bullet$}] if $\widetilde\psi\in C^\infty_c  (0,1)$  and $j\in\mathbb N$, then
$X_0^j \widetilde\psi(x/a^r)=\widetilde\psi_1 (x/a^r)$, for
another   $\widetilde\psi_1\in C^\infty_c  (0,1)$;\smallskip
\item[{\tiny$\bullet$}] $X_0^j a^{-q} = c a^{-q}$, for some constant
  $c$, for all $q>0$ and $j\in\mathbb N$;\smallskip
\item[{\tiny$\bullet$}] if any derivative falls on $\chi$, the
  resulting term is of the form $\Psi\in C^\infty_c(G)$. 
\end{itemize}
Therefore, $X_0^{k_0}X_1^{k_1}g $ is sum of terms of the form
$$
\widetilde\psi(x/a^r) \chi(a) a^{k_1-k_1r-\gamma} +\widetilde\Psi(a,x)
$$
for some $\widetilde\psi\in C^\infty_c  (0,1)$ and $\widetilde\Psi\in
C^\infty_c(G)$.  Thus, $X_0^{k_0}X_1^{k_1}g \in L^p$ if 
\begin{align*}
\int_G |
\widetilde\psi(x/a^r) \chi(a) a^{k_1(1-r)-\gamma} |^p   \frac{\di a
  \di x }{a}
& \lesssim \int_0^1 a^{k_1(1-r)p-\gamma p+r-1} \di a <+\infty\,, 
\end{align*} 
which is the case if and only if $k_1(1-r)p-\gamma p+r>0$; which is
the inequality on the right of
\eqref{cond-gamma-r}.   

On the other hand, 
\begin{align*}
\| g^2\|_p^p
& \approx \int_0^{1/2} a^{-2\gamma p+r-1} \di a\,,
\end{align*}
which is infinite if $\gamma >r/(2p)$.
\end{proof}

\begin{remark}\label{contro-eempio-rem}
{\rm In \cite{V2}  Varopoulos showed that on a Lie group $G$,  $L^p_1$ continuously embeds in
  $L^q(\delta^s)$ if $1\le p\le q<\infty$ and
$s=1-\frac pq$.   
We note that the function $g$ constructed in the theorem
is in $L^p_1$  for $\gamma,r$
  satisfying \eqref{cond-gamma-r}, while, on the other hand $g\not\in
  L^\infty$,  as it is easy to check.     
Thus, this function also shows that the Sobolev embedding
  theorem cannot hold at the limiting point $q=\infty$ and the modular
  function $\delta$ appears in a natural way  in the Sobolev
  embeddings when the group is nonunimodular --- see also \cite{PV}. 
}
\end{remark}

\section{Sobolev norms in the case $\alpha\in (0,1)$}\label{secalpha01}

We shall give two representation formulas for the Sobolev norms when $\alpha\in (0,1)$.  
 
\subsection{A representation formula for the Sobolev norm in terms of $S^{{\rm loc}}_{\alpha}$}
Recall that in \eqref{S-alpha-loc-def} we  have defined the quantity 
$S^{{\rm loc}}_{\alpha}f$.  We now prove Theorem \ref{Salphaloc} (i).
 
 \proof[Proof of Theorem \ref{Salphaloc} (i)]{ \emph{STEP I}.} We shall prove that  
\begin{equation}\label{claim1}
\|\Delta^{\alpha/2} f\|_p
\lesssim   \|S^{{\rm loc}}_{\alpha}f\|_p+\|f\|_p \qquad \forall f\in L^p_{\alpha}\,. 
\end{equation}
We observe that  
$$
\begin{aligned}
\big(g_{1-\alpha/2}\Delta^{\alpha/2}f(x)\big)^2&=\int_0^{\infty}t^{1-\alpha}|\Delta
T_tf(x)|^2\,\di t\\  
&=\int_0^1t^{1-\alpha}   \big| \Delta
T_tf(x)\big|^2  \,\di t+\int_1^{\infty}t^{1-\alpha}   
\big| \Delta T_tf(x)\big|^2  \,\di t\\ 
&=: \CB  \big( g_{1-\alpha/2,0}\Delta^{\alpha/2}f(x)\big)^2
+\big(g_{1-\alpha/2,\infty}\Delta^{\alpha/2}f(x)\big)^2 \,. 
\end{aligned}
$$ 
Notice that  
$$
\big(g_{1-\alpha/2,\infty}\Delta^{\alpha/2}f(x)\big)^2\leq \CB 
\int_1^{\infty}|t\Delta T_tf(x)|^2\,\frac{\di t}{t}\leq \big(g_1f(x)
\big)^2\,,
$$
so that by \eqref{galphap}
$$
\|g_{1-\alpha/2,\infty}\Delta^{\alpha/2}f\|_p\leq \|g_1f\|_p\lesssim \|f\|_p\,.
$$
To estimate $g_{1-\alpha/2,0}\Delta^{\alpha/2}f$ we first notice that
for every $t\in (0,1)$ and $x\in G$, since $\frac{\partial}{\partial
  t}\int_Gp_t\di\rho =0$, we have 
$$
\begin{aligned}
\big| \Delta T_tf(x)\big| & = 
\Big| \frac{\partial}{\partial t}T_tf(x) \Big|=\Big| \frac{\partial
}{\partial t} \Big( \int_Gf(xy^{-1})p_t(y)\di \rho(y)- \int_Gf(x) p_t(y)\di
\rho(y)  \Big)\Big|\\ 
&   \leq    \int_G|f(xy^{-1})-f(x)|  \Big|\frac{\partial
  p_t(y)}{\partial t}\Big| \di \rho(y)\,. 
\end{aligned}
$$ 
Using estimate (iii) in Subsection \ref{pt} for the derivative of the
heat kernel and Cauchy--Schwarz's inequality we have 
% $$
%\begin{aligned}
%\Big| \frac{\partial}{\partial t}T_tf(x) \Big|^2&\lesssim t^{-2}V(\sqrt
%t)^{-2} \Big(\int_{|y|<\sqrt t}|f(xy^{-1})-f(x)|
%\delta^{1/2}(y)e^{-c|y|^2/t} \di \rho(y) \\
%&\quad + \sum_{k=0}^{\infty}
%\int_{2^k\sqrt t\leq |y|<2^{k+1}\sqrt t}|f(xy^{-1})-f(x)|
%\delta^{1/2}(y)e^{-c|y|^2/t} \di \rho(y)  \Big)^2\\
%&\lesssim 
%t^{-2}V(\sqrt
%t)^{-2} \Big(\int_{|y|<\sqrt t}|f(xy^{-1})-f(x)|  \di \rho(y)  \di t\\ 
%&\quad +\sum _{k=0}^{\infty}e^{-c'2^{2k}}\int_0^1 t^{1-\alpha-2} V(\sqrt
%t)^{-2}\Big(\int_{|y|<2^{k+1}\sqrt t}|f(xy^{-1})-f(x)|
%\delta^{1/2}(y)\di \rho(y)\Big)^2
%\end{aligned}
%$$
%Thus
 $$
\begin{aligned}
&\big(g_{1-\alpha/2,0}\Delta^{\alpha/2}f(x)\big)^2\\&\lesssim \int_0^1 t^{1-\alpha-2}  V(\sqrt t)^{-2}
\Big(\int_{|y|<\sqrt t}|f(xy^{-1})-f(x)| \delta^{1/2}(y)  e^{-c_3|y|^2/t} \di \rho(y)\Big)^2 \di t\\ 
&\quad +\sum _{k=0}^{\infty} \int_0^1 t^{1-\alpha-2} V(\sqrt
t)^{-2}\Big( 
\int_{ 2^k <|y|<2^{k+1}\sqrt t } |f(xy^{-1})-f(x)|
\delta^{1/2}(y)   \,  e^{-c_3|y|^2/t}  \di \rho(y)\Big)^2 \di t\\
&\lesssim  \int_0^1 t^{1-\alpha-2}  V(\sqrt t)^{-2}
\Big(\int_{|y|<\sqrt t}|f(xy^{-1})-f(x)|   \di \rho(y)\Big)^2 \di t\\ 
&\quad +\sum _{k=0}^{\infty}e^{-c'2^{2k}}\int_0^1 t^{1-\alpha-2} V(\sqrt
t)^{-2}\Big(\int_{|y|<2^{k+1}\sqrt t}|f(xy^{-1})-f(x)|
\delta^{1/2}(y)\di \rho(y)\Big)^2 \di t\,. 
\end{aligned}
$$
By the change of variables $u=2^{k+1}\sqrt t$ we obtain
$$
\begin{aligned}
&\big(g_{1-\alpha/2,0}\Delta^{\alpha/2}f(x)\big)^2\\  
&\lesssim   \int_0^1 \frac{1}{u^{1+2\alpha}V(u)^2}\Big(\int_{|y|<
  u}|f(xy^{-1})-f(x)|  \di \rho(y)\Big)^2 \di u\\ 
&\quad +\sum _{k=0}^{\infty}e^{-c'2^{2k}}\int_0^{2^{k+1}}
\frac{1}{(2^{-k-1})^{2\alpha}  u^{1+2\alpha}V(2^{-k-1}u)^2  }
\Big(\int_{|y|<u}|f(xy^{-1})-f(x)|  \delta^{1/2}(y)\di \rho(y)\Big)^2
\di u\\ 
&\lesssim \big(S^{{\rm loc}}_{\alpha}f(x)\big)^2\\
&\quad +\sum _{k=0}^{\infty}e^{-c'2^{2k}}\int_0^{1}
\frac{1}{(2^{-k-1})^{2\alpha}  u^{1+2\alpha}V(2^{-k-1}u)^2  }
\Big(\int_{|y|<u}|f(xy^{-1})-f(x)|   \di \rho(y)\Big)^2 \di u\\ 
&\quad +\sum _{k=0}^{\infty}e^{-c'2^{2k}}\int_1^{2^{k+1}}
\frac{1}{(2^{-k-1})^{2\alpha}  u^{1+2\alpha}V(2^{-k-1}u)^2  }
\Big(\int_{|y|<u}|f(xy^{-1})-f(x)|  \delta^{1/2}(y)\di \rho(y)\Big)^2
\di u\,.
\end{aligned}
$$
By \eqref{S-alpha-loc-def} and formula \eqref{pallepiccole} we obtain that
$$
\begin{aligned}
&\big(g_{1-\alpha/2,0}\Delta^{\alpha/2}f(x)\big)^2\\ 
&\lesssim \big(S^{{\rm loc}}_{\alpha}f(x)\big)^2+
\big(S^{{\rm loc}}_{\alpha}f(x)\big)^2 \sum _{k=0}^{\infty}e^{-c'2^{2k}}
(2^{k+1})^{2\alpha+2d} \\ 
&\quad +\sum _{k=0}^{\infty}e^{-c'2^{2k}}   ( 2^{k+1})^{2\alpha+2d}
\int_1^{2^{k+1}}  \frac{1}{  u^{1+2\alpha+2d}  }  |f(x)|^2
\Big(\int_{|y|<u}  \delta^{1/2}(y)\di \rho(y)\Big)^2 \di u\\ 
&\quad +\sum _{k=0}^{\infty}e^{-c'2^{2k}}   ( 2^{k+1})^{2\alpha+2d}
\int_1^{2^{k+1}}  \frac{1}{  u^{1+2\alpha+2d}  }   \Big(\int_{|y|<u}
|f(xy^{-1})|  \delta^{1/2}(y)\di \rho(y)\Big)^2 \di u\\ 
&\lesssim \big(S^{{\rm loc}}_{\alpha}f(x)\big)^2 +\sum
_{k=0}^{\infty}J_k(x)+\sum _{k=0}^{\infty}I_k(x)\,.
\end{aligned}
$$
By \eqref{intdelta12} we deduce that
\begin{equation}\label{Jk}
\begin{aligned}
J_k(x)&\lesssim e^{-c'2^{2k}}   ( 2^{k+1})^{2\alpha+2d}   \int_1^{2^{k+1}}  \frac{1}{  u^{1+2\alpha+2d}  }  |f(x)|^2   u^{2d}e^{2Q u} \di u\\
&\lesssim |f(x)|^2   e^{-c'2^{2k}}   ( 2^{k+1})^{2\alpha+2d}  e^{2Q 2^k}\,,
\end{aligned}
\end{equation}
so that $\sum _{k=0}^{\infty}J_k(x)\lesssim |f(x)|^2  $. We now notice that there exists $c''>0$ such that
$$
\Big\| \Big( \sum_{k=0}^{\infty}I_k\Big)^{1/2}\Big\|_p\lesssim \sum_{k=0}^{\infty} e^{-c''2^{2k}}\Big\| \Big(   \int_1^{2^{k+1}}   \Big(\int_{|y|<u}
|f(xy^{-1})|  \delta^{1/2}(y)\di \rho(y)\Big)^2         \Big)^{1/2}      \di u   \Big\|_p\,.
$$
For every integer $k$, by Minkowski inequality, we get
$$
\begin{aligned}
&\Big(   \int_1^{2^{k+1}}     \Big(\int_G
|f(xy^{-1})| \chi_{B_u}(y)  \delta^{1/2}(y)\di \rho(y)\Big)^2      \di u   \Big)^{1/2} \\&\lesssim 
\int_G \Big(  \int_1^{2^{k+1}}|f(xy^{-1})|^2\chi_{B_u}(y)\delta(y)\di u \Big)^{1/2}\di\rho(y)\\
&\lesssim \int_{B_1}|f(xy^{-1})|\delta^{1/2}(y)   \Big(  \int_1^{2^{k+1}} \di u \Big)^{1/2} \di\rho(y)\\
&\qquad +\int_{1\leq |y|\leq 2^{k+1}}|f(xy^{-1})|\delta^{1/2}(y)  \Big(  \int_{|y|}^{2^{k+1}}\di u \Big)^{1/2}\di\rho(y)\\
&\lesssim \int_{B_1}|f(xz)| 2^{k/2}\di\lambda(z)+\int_{1\leq |z|\leq 2^{k+1}}|f(xz)|\delta^{1/2}(z^{-1})2^{k/2}\di\lambda(z)\,.
\end{aligned}
$$
We then obtain, by applying once again Minkowski inequality, 
$$
\begin{aligned}
 &\Big\|   \Big(   \int_1^{2^{k+1}}    \Big(\int_G
|f(xy^{-1})| \chi_{B_u}(y)  \delta^{1/2}(y)\di \rho(y)\Big)^2       \di u    \Big)^{1/2}      \Big\|_p\\
&\lesssim 2^{k/2}\,  \int_{B_1}    \Big(\int_G|f(xz)|^p \di\rho(x)\Big)^{1/p}  \di\lambda(z)+ 2^{k/2}\,    \int_{1\leq |z|\leq 2^{k+1}}   \Big(\int_G|f(xz)|^p   \di\rho(x)\Big)^{1/p}  \delta^{-1/2}(z)\di\lambda(z)\\
&\lesssim 2^{k/2}\|f\|_p+2^{k/2}\|f\|_p\int_{B_{2^{k+1}}}\delta^{1/2}\di\rho\\
&\lesssim 2^{k/2+kd}\|f\|_p \,e^{Q2^k}\,,
\end{aligned}
$$
where we have applied \eqref{intdelta12}. We then have
\begin{equation}\label{Ik}
\begin{aligned}
\Big\| \Big( \sum_{k=0}^{\infty}I_k\Big)^{1/2}\Big\|_p\lesssim \sum_{k=0}^{\infty} e^{-c''2^{2k}}2^{k/2+kd}\|f\|_p e^{Q2^k}\lesssim \|f\|_p\,.
\end{aligned}
\end{equation}
In conclusion, by \eqref{Jk} and \eqref{Ik} we get
$$
\|g_{1-\alpha/2,0}\Delta^{\alpha/2}f\|_p\lesssim \|S^{{\rm loc}}_{\alpha}f\|_p+\|f\|_p\,,
$$
as required. 
 
 \smallskip
 
{ \emph{STEP II}.} We shall prove that  
\begin{equation}\label{claim2}
 \|S^{{\rm loc}}_{\alpha}f\|_p\lesssim \|\Delta^{\alpha/2} f\|_p+\|f\|_p  \qquad \forall f\in L^p_{\alpha}\,. 
\end{equation}
To prove it we write $f=(f-T_1f)+T_1f$ and we estimate $\|S^{{\rm loc}}_{\alpha}(f-T_1f)\|_p$ and $\|S^{{\rm loc}}_{\alpha}T_1f\|_p$, separately. 

Arguing as in \cite{CRTN} we write 
$$%\begin{equation}\label{fm-gm}
f-T_1f=\sum_{m=-\infty}^{-1}f_m\,,\ {\rm{where}} \ 
f_m=-\int_{2^m}^{2^{m+1}}\frac{\partial}{\partial t}T_tf \di t\,,
\quad {\rm{and}} \quad g_m=\int_{2^{m-1}}^{2^m}
\Big|\frac{\partial}{\partial t}T_tf \Big|\di t 
\,. 
$$% \end{equation}  
We then obtain
$$
\begin{aligned}
& \big(S^{{\rm loc}}_{\alpha}(f-T_1f)(x)\big)^2\\
&=\int_0^1
\frac{1}{u^{1+2\alpha}V(u)^2}  \Big(\int_{|y|<u}| (f-T_1f)(xy^{-1})-
(f-T_1f)(x)|\di\rho(y)\Big)^2\di u\\ 
&=\sum_{j=-\infty}^{-1} \int_{2^j}^{2^{j+1}}
\frac{1}{u^{1+2\alpha}V(u)^2}  \Big(\int_{|y|<u}| (f-T_1f)(xy^{-1})-
(f-T_1f)(x)|\di\rho(y)\Big)^2\di u\\ 
&\lesssim \sum_{j=-\infty}^{-1}   \frac{2^j}{ (2^j)^{1+2\alpha+2d} }
\Big(\int_{|y|<2^{j+1}}| (f-T_1f)(xy^{-1})-
(f-T_1f)(x)|\di\rho(y)\Big)^2  \,,
\end{aligned}
$$
where we applied \eqref{pallepiccole}. Notice that 
$$
\begin{aligned}
& \frac{1}{ 2^{jd} }
 \int_{|y|<2^{j+1}}| (f-T_1f)(xy^{-1})-
(f-T_1f)(x)|\di\rho(y)\\
&\lesssim   \sum_{m=-\infty}^{-1}   \frac{1}{ 2^{jd} }\int_{|y|<2^{j+1}}|
f_m(xy^{-1})-  f_m(x)|\di\rho(y)  \,.
\end{aligned}
$$

If $m< 2j+3$, then 
\begin{equation}\label{m<}
 \frac{1}{ 2^{jd} }\int_{|y|<2^{j+1}}| f_m(xy^{-1})-  f_m(x)|\di\rho(y) \lesssim {M}^{1}g_{m+1}(x)\,,
 \end{equation}
 where $M^1$ is the local maximal function defined in \eqref{MR}.  In
 order to treat the case when $m\geq 2j+3$, we notice that for every
 $j\leq -1$, $y\in B_{2^{j+1}}$ and $x\in G$    
\begin{equation}\label{teovalormedio}
| f_m(xy^{-1})-  f_m(x)|\leq 2^{j+1}\,\sup \{|X_if_m(w)|: i=1,\dots,q,\,|w^{-1}x|\leq 2^{j+1}\}\,.
\end{equation}
Since 
$$
f_m=-2\int_{2^{m-1}}^{2^m}   \frac{\partial}{\partial t}(T_{2t}f)\di t=-4 \int_{2^{m-1}}^{2^m}T_t\frac{\partial}{\partial t}(T_{t}f)\di t\,,
$$
by applying the estimates of the heat kernel given in Subsection \ref{pt}, for every $w$ such that $|w^{-1}x|\leq 2^{j+1}$ we have 
$$ 
\begin{aligned}
|X_if_m(w)|&\lesssim \int_{2^{m-1}}^{2^m} \int_G \Big| \frac{\partial}{\partial t}(T_{t}f)(z)\Big| |X_ip_t(z^{-1}w)|\di\lambda(z)\\
&\lesssim \int_G \int_{2^{m-1}}^{2^m}  \Big| \frac{\partial}{\partial t}(T_{t}f)(z)\Big| t^{-1/2}V(\sqrt t)^{-1}\,\delta^{1/2}(z^{-1}w) e^{ -c|z^{-1}w|^2/t }\di t\di\lambda(z)\\
&\lesssim 2^{-m/2}2^{-md/2}\int_Gg_m(z)\delta^{1/2}(z^{-1}x)e^{-c|z^{-1}x|^2/2^m}\di\lambda (z)\\
&\lesssim 2^{-m/2}T_{c2^m}g_m(x)\,,
\end{aligned}
$$
for a suitable constant $c$. From \eqref{teovalormedio}  if
follows that 
\begin{equation}\label{altrocaso}
\begin{aligned}
 \frac{1}{ 2^{jd} }\int_{|y|<2^{j+1}}| f_m(xy^{-1})-
 f_m(x)|\di\rho(y)
& \lesssim  \frac{1}{ 2^{jd} }2^{j+1}2^{-m/2}T_{c2^m}g_m(x)2^{jd}\\
& \lesssim 2^{j-m/2}T_{c2^m}g_m(x)\,.
 \end{aligned}
\end{equation}
Thus, by \eqref{m<} and \eqref{altrocaso} 
 $$
\begin{aligned}
\big(S^{{\rm loc}}_{\alpha}(f-T_1f)(x)\big)^2&\lesssim \sum_{j=-\infty}^{-1}   \frac{1}{ (2^j)^{2\alpha } } \Big(\sum_{m=-\infty}^{2j+3-1} {M}^{1}g_{m+1}(x)+\sum_{m=2j+3}^{-1}         2^{j-m/2}T_{c2^m}g_m(x)\Big)^2 \,.
%&\lesssim \Big\|  \Big(\sum_{m=-\infty}^{-1} 2^{-m\alpha}g_m^2\Big)^{1/2} \Big\|_p\,,
\end{aligned}
$$
We can argue as in \cite[p.298-303, 308-309]{CRTN} to deduce that   
 $$
\begin{aligned}
\Big\| S^{{\rm loc}}_{\alpha}(f-T_1f)(x)\Big\|_p 
&\lesssim \Big\|  \Big(\sum_{m=-\infty}^{-1} 2^{-m\alpha}g_m^2\Big)^{1/2} \Big\|_p\,.
\end{aligned}
$$
Since
$$
\sum_{m=-\infty}^{-1} 2^{-m\alpha}g_m^2(x)\lesssim \sum_{m=-\infty}^{-1}  \int_{2^{m-1}}^{2^m}\Big| \frac{\partial}{\partial t}T_tf(x)\Big|^2\,\di t\lesssim g_{1-\alpha/2} \Delta^{\alpha/2}f(x)^2\,,
$$ 
we have
\begin{equation}\label{f-T1f}
\| S^{{\rm loc}}_{\alpha}(f-T_1f) \|_p\lesssim \|g_{1-\alpha/2} \Delta^{\alpha/2}f  \|_p\lesssim \|  \Delta^{\alpha/2}f  \|_p\,.
\end{equation}
In order to estimate the norm of $S^{{\rm loc}}_{\alpha}T_1f$ we first notice that for every $x\in G$ and $y\in B_1$ 
$$
\begin{aligned}
|T_1f(xy^{-1})-T_1f(x)|&\lesssim |y|\sup\{ |X_iT_1f(w)|: |w^{-1}x|\leq |y|\}\\
&\lesssim |y|\sup\{ |X_iT_1f(w)|: |w^{-1}x|\leq 1\}\,.\end{aligned}
$$
By the estimates of the heat kernel and its derivatives in Subsection \ref{pt} there exists $t_0>0$ such that for every $w$ such that  $|w^{-1}x|\leq 1$  
$$
\begin{aligned}
|X_iT_1f(w)|&=|f\ast X_ip_1(w)|\\
&\leq \int |f(wv^{-1})||X_ip_1(v)|\di\rho\\
&\lesssim \int |f(wv^{-1})|  \delta^{1/2}(v) e^{-c|v|^2}  \di\rho(v)\\
&\lesssim \int |f(wv^{-1})|   p_{t_0}(v) \di\rho(v)\\
%&= \int |f(wv)|   p_{t_0}(v^{-1})  \di\lambda(v)\\
&= \int |f(z)|   p_{t_0}(z^{-1}w)  \di\lambda(z)\\
%&\lesssim \int |f(z)|   \delta^{1/2}(z^{-1}w) \,e^{-c|z^{-1}w|^2} \di\lambda(z)\\
%&\lesssim \int |f(z)|   \delta^{1/2}(z^{-1}x) \,e^{-c|z^{-1}x|^2} \di\lambda(z)\\
&\lesssim \int |f(z)|   p_{t_0}(z^{-1}x)  \di\lambda(z)\\ 
&=T_{t_0}|f|(x)\,.
\end{aligned}
$$
Thus
$$
\begin{aligned}
S^{{\rm loc}}_{\alpha}T_1f(x)^2&=\int_0^1\Big(  \frac{1}{u^{\alpha} V(u)}\int_{|y|<u}|T_1f(xy^{-1})-T_1f(x)|\di\rho(y) \Big)^2\frac{\di u}{u}\\
&\lesssim \int_0^1   \frac{1}{u^{2\alpha} V(u)^2} \Big(\int_{|y|<u} u T_{t_0}|f|(x)  \di\rho(y) \Big)^2\frac{\di u}{u}\\
&\lesssim T_{t_0}|f|(x)^2\,,
\end{aligned}
$$
where we used the fact that $\alpha\in (0,1)$. 
\begin{equation}\label{T1f}
\|S^{{\rm loc}}_{\alpha}T_1f\|_p\lesssim \|T_{t_0}|f|\|_p\lesssim
\|f\|_p\,,  
\end{equation}
which together with \eqref{f-T1f} gives \eqref{claim2}, as required. \qed

\medskip

 The representation formula that we just proved is the key ingredient to show the following lemma, which will be useful to prove the "interpolation estimate" given in Proposition \ref{int} below. 
 \begin{lemma}\label{normamodulof}
 For all $\alpha\in (0,1)$ and $p\in (1,\infty)$
 $$
 \| |f|\|_{\alpha,p}\lesssim \|f\|_{\alpha,p}\,.
 $$
 \end{lemma}
 \begin{proof}
 It suffices to notice that, for every $x\in G$, $S^{{\rm loc}}_{\alpha}(|f|)(x)\leq
 S^{{\rm loc}}_{\alpha}(f)(x)$, and use the representation of the
 $L^p_{\alpha}$-norm given by Theorem \ref{Salphaloc}. 
 \end{proof}
 \begin{proposition}\label{int}
 Let $\alpha,\beta,\gamma\geq 0$, $1<p,r<\infty$, $1<q\leq \infty$ and
 $0<\theta<1$ be such that $\gamma=\theta \alpha+(1-\theta)\beta$ and
 $1/r=\theta/p+(1-\theta)/q$. Then for all $f\in L^p_{\alpha}\cap L^q_{\beta}$ 
 \begin{equation}\label{Holder}
 \|f\|_{\gamma,r}\lesssim \|f\|_{\alpha,p}^{\theta}\|f\|_{\beta,q}^{1-\theta}\,.
 \end{equation}
 \end{proposition}
 \begin{proof}
 Notice that it is enough to give the proof in the case when $\beta=0$. 
 
 We then take $\alpha,\gamma> 0$, $1<p,r<\infty$, $1<q\leq \infty$ and
 $0<\theta<1$ such that $\gamma=\theta \alpha$ and
 $1/r=\theta/p+(1-\theta)/q$. By \eqref{Sobolev-norm} we
 have $ 
 \|f\|_{\gamma,r}=\|f\|_r+\|\Delta^{\gamma/2}f\|_r$. Choose $a,b,s$
 such that $a+b=r$, $as=p$ and $bs'=q$. By H\"older's inequality we
 obtain that  
 \begin{equation}\label{Holderf}
 \|f\|_r\leq  \Big(\int_G|f|^{as}
 \di\rho\Big)^{1/sr}\,\Big(\int_G|f|^{bs'} \di\rho\Big)^{1/s'r}=
 \|f\|_p^{a/r}\|f\|_q^{b/r}=\|f\|_{p}^{\theta}\|f\|_{q}^{1-\theta}\leq
 \|f\|_{\alpha, p}^{\theta} \|f\|_{q}^{1-\theta}\,.
 \end{equation}
 It remains to estimate $\|\Delta^{\gamma/2}f\|_r$. 
 
 If $q<\infty$, choose $\delta>0$ and $k\in\mathbb N$ such that
 $\delta+\gamma/2=k$. Then if $a=2\theta (k-\frac12-\frac{\alpha}{2})$
 and $b=2(1-\theta)(k-\frac12)$, by applying H\"older's inequality in
 $t$ we get 
 \begin{equation}\label{gdelta}
 \begin{aligned}
 (g_{\delta}\Delta^{\gamma/2}f(x))^2&=\int_0^{+\infty}    t^{a+b}|\Delta^kT_tf(x)|^2\di t\\
 &\leq (g_{k-\alpha/2}\Delta^{\alpha/2}f(x))^{2\theta}\,(g_kf(x))^{2(1-\theta)}\,.
 \end{aligned}
 \end{equation}
 Therefore, by Littlewood-Paley-Stein theory, \eqref{gdelta} 
and applying H\"older's inequality in the $x$-variable
 \begin{equation}\label{qfinito}
 \begin{aligned}
 \|\Delta^{\gamma/2}f\|_r&\approx \|g_{\delta}\Delta^{\gamma/2}f  \|_r\\
 &\leq \Big(   \int_G  (g_{k-\alpha/2}\Delta^{\alpha/2}f(x))^{\theta\,r}\,(g_kf(x))^{(1-\theta)\,r} \di\rho(x)\Big)^{1/r}\\
 &\leq \Big(   \int_G
 (g_{k-\alpha/2}\Delta^{\alpha/2}f(x))^{\theta\,rs}
 \di\rho(x)\Big)^{1/sr} \Big(   \int_G  (g_kf(x))^{(1-\theta)\,rs'}
 \di\rho(x)\Big)^{1/s'r}\\ 
 &= \|g_{k-\alpha/2}\Delta^{\alpha/2}f\|_p^{\theta} \,\|g_kf\|_q^{1-\theta}\\
 &\lesssim \|\Delta^{\alpha/2}f\|_p^{\theta}\|f\|_q^{1-\theta}\\
&\lesssim \|f\|_{\alpha, p}^{\theta}\|f\|_{q}^{1-\theta} \,.
 \end{aligned}
 \end{equation}
 Estimates \eqref{Holderf} and \eqref{qfinito} prove the proposition
 in the case when $q$ is finite and $\beta=0$.  
 
 Suppose now that $q=\infty$. We follow closely \cite[Theorem 2.4]{AM} using a complex interpolation argument. 
 
 Assume first that $f$ and $h$ are nonnegative simple functions and define for $z\in \Sigma^{\infty}_0=\{z\in\mathbb C:   \Re z\geq 0\}$ 
 $$
 w(z)=\int_G\Delta^{-\alpha z/2}f(x) h(x)^{(1-1/p)(1-z)+z}\di \rho(x)\,.
 $$
The function $w$ is continuous in $\Sigma^{\infty}_0$, holomorphic in the interior of $\Sigma^{\infty}_0$ and bounded in any strip $\Sigma^{c}_0=\{z\in\mathbb C:   0\leq  \Re z\leq c\}$, $c\in\mathbb R^+$. When $z=i\zeta,\,\zeta\in\mathbb R$, by \cite{M} there exists a positive constant $C_p$ such that
$$
\|\Delta^{-i\alpha \zeta/2}f\|_p\leq C_p   (1+|\alpha \zeta|/2)^{1/2}\,e^{    \frac{\pi}{2}  \frac{|\alpha \zeta|}{2} }\|f\|_p\,,
$$ 
so that
$$
|w(i\zeta)|\leq C_p(1+|\alpha \zeta|/2)^{1/2}\,  e^{    \frac{\pi}{2}  \frac{|\alpha \zeta|}{2} }   \|f\|_p\|h\|_1^{1-1/p}\,.
$$
On the other hand, since $f$ is nonnegative
$$
\begin{aligned}
|\Delta^{-\alpha/2-i\alpha\zeta/2}f(x)|&=\frac{1}{|\Gamma(\alpha/2+i\alpha\zeta/2)|}\Big|\int_0^{\infty}t^{\alpha/2+i\alpha\zeta/2-1} e^{-t\Delta}f(x)\di t \Big|\\
&\leq \frac{1}{|\Gamma(\alpha/2+i\alpha\zeta/2)|} \int_0^{\infty}t^{\alpha/2-1} e^{-t\Delta}f(x)\di t \\
&= \frac{\Gamma(\alpha/2)}{|\Gamma(\alpha/2+i\alpha\zeta/2)|}\Delta^{-\alpha/2}f(x)\,.
\end{aligned}
$$
Thus
$$
|w(1+i\zeta)|\leq \frac{\Gamma(\alpha/2)}{|\Gamma(\alpha/2+i\alpha\zeta/2)|} \|\Delta^{-\alpha/2}f\|_{\infty}\,\|h\|_1\,.
$$
Define $W(z)=\Gamma(1+\alpha z/2) \frac{1}{1+z}w(z)$ for $z\in \Sigma^{\infty}_0$. By the estimates satisfied by the function $w$ on the boundary of the strip $\Sigma_0^1$ and the three lines theorem we get
$$
|W(1-\theta)|\lesssim 
\|f\|_p^{\theta}\|\Delta^{-\alpha/2}f\|_{\infty}^{1-\theta}\|h\|_1^{(1-1/p)\theta+1-\theta}\,,
$$
which implies that
$$
|w(1-\theta)|\lesssim \|f\|_p^{\theta}\|\Delta^{-\alpha/2}f\|_{\infty}^{1-\theta}\|h\|_1^{1/r'}\,.
$$
By taking the supremum over all functions $g=h^{1/r'}$ such that $\|g\|_{r'}\leq 1$ we obtain that 
$$
\|\Delta^{-(1-\theta)\alpha/2}f\|_r\lesssim \,\|f\|_p^{\theta}\|\Delta^{-\alpha/2}f\|_{\infty}^{1-\theta}
$$
for all nonnegative functions $f$. This implies that 
$$
\|\Delta^{\gamma/2}g\|_r\lesssim\,\|\Delta^{\alpha/2}g\|_p^{\theta}\|g\|_{\infty}^{1-\theta}
$$
for all nonnegative functions $g$. By \eqref{Holderf} and the estimate above we deduce that for all nonnegative functions
\begin{equation}\label{fnonnegative}
\|g\|_{\gamma,r}\lesssim\,\|g\|_{\alpha,p}^{\theta}\|g\|_{\infty}^{1-\theta}\,.
\end{equation}
Take now $\alpha\in [0,1]$ and $f$ of arbitrary sign. Then writing $f=f_+-f_-$, applying \eqref{fnonnegative} to $f_+$ and $f_-$, using Lemma \ref{normamodulof} and noticing that $\|\Delta^{\alpha/2}f_{\pm}\|_p\lesssim \|\Delta^{\alpha/2}f\|_p$ we obtain that
\begin{equation}\label{alpha<=1}
\|f\|_{\gamma,r}\lesssim \,\|f\|_{\alpha,p}^{\theta}\|f\|_{\infty}^{1-\theta}\,.
\end{equation}
It remains to consider the case when $\alpha>1$. Suppose first that $\gamma<1$, $\theta\in (0,1)$, $\alpha>1$, $\gamma=\theta\alpha$ and $\frac{1}{r}=\frac{\theta}{p}$. We choose $\beta<1$ such that $\gamma<\beta<\gamma r$ and $s>1$ such that $\frac{1}{r}=\frac{\gamma}{\beta s}$. Then by \eqref{alpha<=1}  
$$
\|f\|_{\gamma,r}\lesssim\,\|f\|_{\beta,s}^{\tilde \theta}\|f\|_{\infty}^{1-\tilde\theta}\,,
$$
for $\tilde \theta\in (0,1)$ such that $\gamma=\tilde\theta\beta$ and $\frac1r=\frac{\tilde\theta}{s}$. Moreover, by \eqref{Holder}
$$
\|f\|_{\beta,s}\lesssim \|f\|_{\alpha,p}^{\theta'}\|f\|_{\gamma,r}^{1-\theta'}\,,
$$
for $\theta'\in (0,1)$ such that $\beta=\theta'\alpha+(1-\theta')\gamma$, and $\frac1s=\frac{\theta'}{p}+\frac{1-\theta'}{r}$. Putting together the two estimates above we obtain 
\begin{equation}\label{gamma<1}
\|f\|_{\gamma,r}\lesssim \|f\|_{\alpha,p}^{\theta}\|f\|_{\infty}^{1-\theta}\,,
\end{equation}
which proves the theorem for $\alpha>1$ and $\gamma<1$. Take now $\gamma\geq 1$ and choose $q$ such that $\gamma r<q<r$ and $\beta<1$ such that $\beta q=\gamma r$. We have
$$
\|f\|_{\gamma,r} \lesssim \|f\|_{\alpha,p}^{\tilde \theta}\|f\|_{\beta,q}^{1-\tilde\theta}\,,
$$
for $\gamma=\tilde\theta\alpha+(1-\tilde\theta)\beta$ and
$\frac{1}{r}=\frac{\tilde\theta}{p}+\frac{1-\tilde\theta}{q}$. By
\eqref{gamma<1} we get 
$$
\|f\|_{\beta,q}\lesssim \|f\|_{\alpha,p}^{\theta'}\|f\|_{\infty}^{1-\theta'}\,,
$$
where $\beta=\theta'\alpha$ and $\frac1q=\frac{\theta'}{p}$.  Putting
together the two estimates above we obtain  
\begin{equation}\label{gamma>=1}
\|f\|_{\gamma,r}\lesssim \|f\|_{\alpha,p}^{\theta}\|f\|_{\infty}^{1-\theta}\,,
\end{equation} 
which proves the theorem for $\alpha>1$, $\gamma\geq 1$, $q=\infty$
and $\beta=0$. The proof of the proposition is now complete.  
 \end{proof}
 
 \subsection{A representation formula for the Sobolev norm in terms of
   $D^{{\rm loc}}_{\alpha}$} 
Recall that in \eqref{D-alpha-loc-def} we  have defined the quantity 
$D^{{\rm loc}}_{\alpha}f$.  We shall prove Theorem \ref{Salphaloc}
(ii). To do so, we first need some tools and some technical results
that we shall introduce below.

For every locally integrable function $f$ and every ball $B$ we denote
by $f_B$ the average $\frac{1}{\rho(B)}\int_Bf\di\rho$. For every
$q\in [1,\infty), r>0, x\in G$ we define  
$$
\Omega^{(q)}_f(x,r)=\sup\Big\{\Big(
\frac{1}{\rho(B)}\int_B|f-f_B|^q\di\rho\Big)^{1/q}:\, B\in\mathcal B_r,\, x\in B  \Big\}\,, 
$$
and 
$$
\Omega^{\infty}_f(x,r)=\sup\{ \|f-f_B\|_{\infty} :\, B\in\mathcal B_r,\,  x\in B  \}\,.  
$$
 We recall that $\mathcal B_R$ denotes the collection of balls of
radius $\le R$.   We simply write $
\Omega_f(x,r)$ for $\Omega^{(1)}_f(x,r)$. 
\begin{lemma}\label{Omega}
For every locally integrable function $f$ the following hold:
\begin{itemize}
\item[(i)] $\Omega_f(x,r)\leq \Omega^{(q)}_f(x,r)$ for every $q\in [1,\infty], r>0$;
\item[(ii)] if $B, B'\in\mathcal B_1$ and $B\subset B'$, then 
$$
 \frac{1}{\rho(B)}\int_B|f-f_B|\di\rho \leq 2  \frac{\rho(B')}{\rho(B)}    \frac{1}{\rho(B')}\int_{B'}|f-f_{B'}|^q\di\rho\,;
$$
\item[(iii)] for every $x\in G$, $r\leq s\leq 2r\leq 2$ 
$$
\Omega_f(x,r)\lesssim \Omega_f(x,s) \lesssim \Omega_f(x,2r)\,;
$$
\item[(iv)] for every $B\in \mathcal B$ of radius $r$ and almost every $y\in B$ 
$$
|f(y)-f_B|\lesssim \int_0^{8r} \Omega_f(y,s)\frac{\di s}{s}\,.
$$
\end{itemize}
\end{lemma}
The above lemma was proved in \cite{CRTN}: the same proof works in our setting, since only the local doubling property plays a role here. 

For every locally integrable function $f$, $q\in[1,\infty]$, $R>0$, $\alpha\in (0,1)$ and $x\in G$ we define 
$$
\begin{aligned}
G^{{\rm loc}}_{\alpha}f(x)&=\Big( \int_0^1[r^{-\alpha} \Omega_f(x,r)]^2  \frac{\di r}{r}\Big)^{1/2}\\
G^{{\rm loc}}_{\alpha,q}f(x)&=\Big( \int_0^1[r^{-\alpha} \Omega^{(q)}_f(x,r)]^2   \frac{\di r}{r}\Big)^{1/2}\\
G^{R}_{\alpha}f(x)&=\Big( \int_0^R[r^{-\alpha} \Omega_f(x,r)]^2   \frac{\di r}{r}\Big)^{1/2}\\
G^{R}_{\alpha,q}f(x)&=\Big( \int_0^R[r^{-\alpha} \Omega^{(q)}_f(x,r)]^2   \frac{\di r}{r}\Big)^{1/2}\\
S^{R}_{\alpha}f(x)&=\Big( \int_0^R \Big[   \frac{1}{u^{\alpha}V(u)}\int_{|y|<u}|f(xy^{-1})-f(x)|\di\rho(y) \Big]^2\frac{\di u}{u}  \Big)^{1/2}\,.
\end{aligned}
$$
\begin{lemma}\label{R1R2}
For every $R_1, R_2>0$, $p\in(1,\infty)$ and $\alpha\in(0,1)$, 
$$
\begin{aligned}
\|S^{R_1}_{\alpha}f\|_p+\|f\|_p &\approx \|S^{R_2}_{\alpha}f\|_p+\|f\|_p\,,\\
\|G^{R_1}_{\alpha}f\|_p+\|f\|_p &\approx \|G^{R_2}_{\alpha}f\|_p+\|f\|_p\,.
\end{aligned}
$$
\end{lemma}
\begin{proof}
Assume $R_1\leq R_2$. Then it is obvious that $S^{R_1}_{\alpha}f(x)\leq S^{R_2}_{\alpha}f(x)$ and $G^{R_1}_{\alpha}f(x) \leq G^{R_2}_{\alpha}f(x)$ for every $x\in G$, so that 
$$
\begin{aligned}
\|S^{R_1}_{\alpha}f\|_p+\|f\|_p &\leq \|S^{R_2}_{\alpha}f\|_p+\|f\|_p\\
\|G^{R_1}_{\alpha}f\|_p+\|f\|_p &\leq \|G^{R_2}_{\alpha}f\|_p+\|f\|_p\,.
\end{aligned}
$$
In order to prove the lemma, using \eqref{pallainx0}, 
 we notice that
$$
\begin{aligned}
&\int_{R_1}^{R_2}\Big[ \frac{1}{r^{\alpha}V({r}) }\int_{|y|\leq r} |f(xy^{-1})-f(x)|\di\rho(y) \Big]^2\frac{\di r}{r}\\
&=\int_{R_1}^{R_2}\Big[  \frac{1}{r^{\alpha}V({r}) }\int_{|y|\leq r} |f(xy)-f(x)|\di\lambda(y) \Big]^2\frac{\di r}{r}\\
&\lesssim \int_{R_1}^{R_2}\Big[     \frac{1}{r^{\alpha}V({r}) }\int_{B(x,r)} |f(z)|\di\lambda(z)+  \frac{|f(x)|}{r^{\alpha} } \Big]^2\frac{\di r}{r}\\
&= \int_{R_1}^{R_2}\Big[     \frac{1}{r^{\alpha}V({r}) }\int_{B(x,r)} |f(z)|  \delta(z)\di\rho(z)+  \frac{|f(x)|}{r^{\alpha} } \Big]^2\frac{\di r}{r}\\
&\lesssim  \int_{R_1}^{R_2}\Big[     \frac{1}{r^{\alpha}V({r}) }\int_{B(x,r)} |f(z)|  \delta(x)\di\rho(z)+  \frac{|f(x)|}{r^{\alpha} } \Big]^2\frac{\di r}{r}\\
&\lesssim  \int_{R_1}^{R_2}\Big[  \frac{1}{r^{\alpha}} M^{R_2}f(x)+\frac{|f(x)|}{r^{\alpha} } \Big]^2\frac{\di r}{r}\\&\lesssim M^{R_2}f(x)+|f(x)|\,,
\end{aligned}
$$
 where $M^{R_2}$ is the local maximal function defined in \eqref{MR}.  It follows that 
$$
\|S^{R_2}_{\alpha}f\|_p+\|f\|_p\leq \|S^{R_1}_{\alpha}f\|_p+\|f\|_p+\|M^{R_2}f\|_p+\|f\|_p\lesssim \|S^{R_1}_{\alpha}f\|_p+\|f\|_p\,,
$$
where we have applied the boundedness of the maximal function $M^{R_2}$ on $L^p$. A similar argument shows that for every $r\leq R_2$, $\Omega_f(x,r)\leq M^{R_2}f(x)$ so that 
$$
\int_{R_1}^{R_2} [r^{-\alpha} \Omega_f(x,r)]^2  \frac{\di r}{r}\lesssim \int_{R_1}^{R_2}[   M^{R_2}f(x)  ]^2  \frac{\di r}{r^{1+\alpha}}\,.
$$
Thus
$$
\|G^{R_2}_{\alpha}f\|_p+\|f\|_p\lesssim \|G^{R_1}_{\alpha}f\|_p+\|f\|_p+\|M^{R_2}f\|_p\lesssim \|G^{R_1}_{\alpha}f\|_p+\|f\|_p\,,
$$
applying again the $L^p$ boundedness of the maximal function $M^{R_2}$. 
\end{proof}

\begin{proposition}\label{GR}
For every locally integrable function $f$, $\alpha\in (0,1)$ $p\in (1,\infty)$ and $R_1, R_2>0$ the following hold:
\begin{itemize} 
\item[(i)] if $p\in (1,2]$, $\alpha\leq\frac{d}{p}$, $q<\frac{dp}{d-\alpha p}$, then
$$
\|G^{R_1}_{\alpha,q}f\|_p+\|f\|_p \approx \|G^{R_2}_{\alpha}f\|_p+\|f\|_p\,;
$$
\item[(ii)] if $p\in [2,\infty)$, $\alpha\leq\frac{d}{p}$, $q<\frac{2d}{d-2\alpha}$, then
$$
\|G^{R_1}_{\alpha,q}f\|_p+\|f\|_p \approx \|G^{R_2}_{\alpha}f\|_p+\|f\|_p\,;
$$
\item[(iii)] if $p\in (1,\infty)$, $\alpha>\frac{d}{p}$, then
$$
\|G^{R_1}_{\alpha,\infty}f\|_p+\|f\|_p \approx \|G^{R_2}_{\alpha}f\|_p+\|f\|_p\,.
$$ 
\end{itemize} 
\end{proposition}
\begin{proof}
In view of Lemma \ref{R1R2} we can assume $R_1=R_2=R$. We shall prove
(i). The proofs of statements (ii-iii) are similar and we
omit them. 

By Lemma \ref{Omega}(i) it follows that for every $p\in (1,\infty)$ and $q>1$, $\|G^{R}_{\alpha}f\|_p\lesssim \|G^{R}_{\alpha,q}f\|_p$.  

For every ball $B'=B(y',s)$ with $s<1$, and every $y\in B'$ we have that
$$
\int_{B'}|f-f_{B'}|\di\rho\leq\int_{B'}\Omega_f(z,s)\di\rho(z)\leq \int_{B(y,2s)}\Omega_f(z,s)\di \rho(z)\,,
$$ 
so that by the local doubling property  
$$
\frac{1}{\rho(B')}\int_{B'}|f-f_{B'}|\di\rho\lesssim \frac{1}{\rho(B(y,2s))}\int_{B(y,2s)}\Omega_f(z,s)\di \rho(z)\,.
$$
It follows that
$$
\Omega_f(y,s)\lesssim \frac{1}{\rho(B(y,2s))}\int_{B(y,2s)}\Omega_f(z,s)\di \rho(z)\,.
$$
By Lemma \ref{Omega}(iv) for every ball $B=B(y_0,r)$, $r<1$, for almost every $y\in B$ and for every $\beta\geq 0$
$$
\begin{aligned}
|f(y)-f_B|&\lesssim \int_0^{8r} \frac{1}{\rho(B(y,2s))}\int_{B(y,2s)}\Omega_f(z,s)\di \rho(z)\frac{\di s}{s}\\
&= \int_0^{8r} \frac{\rho(B(y,2s))^{-\beta}}{\rho(B(y,2s))^{1-\beta}}\int_{B(y,2s)}\Omega_f(z,s) \chi_{3B}(z)\di \rho(z)\frac{\di s}{s}\\
&\lesssim \int_0^{8r} [\delta^{-1}(y)V(2s)]^{-\beta}M^1_{\beta}\big( \Omega_f(\cdot,s) \chi_{3B}\big)(y)\frac{\di s}{s}\\
&\lesssim \delta(y_0)^{\beta}\,\int_0^{8r}  V(s)^{-\beta}M^1_{\beta}\big( \Omega_f(\cdot,s) \chi_{3B}\big)(y)\frac{\di s}{s}\,,
\end{aligned}
$$
where we applied again the local doubling property  and where
$M^1_{\beta}$ is the modified local maximal function defined in
\eqref{MRbeta}. Take now $p\in (1,2]$, $\alpha\leq \frac{d}{p}$,
$q<\frac{dp}{d-p\alpha}$. We choose $p_0<p$ and $0\leq \beta
<\frac{\alpha}{d}\leq \frac{1}{p}<1$ such that
$\frac{1}{q}=\frac{1}{p_0}-\beta$. Then $M^1_{\beta}$ is bounded from
$L^{p_0}$ to $L^q$. This implies that 
$$
\begin{aligned}
& \frac{1}{\rho(B)^{1/q}}\|f-f_B\|_{L^q(B)} \\
&\lesssim \frac{1}{\rho(B)^{1/q}}  \delta(y_0)^{\beta}\,\int_0^{8r}  V(s)^{-\beta}  \| \Omega_f(\cdot,s)  \|_{L^{p_0}(3B)}\frac{\di s}{s}\\
&\leq  \frac{1}{\rho(B)^{1/q}}  \delta(y_0)^{\beta}\,\rho(3B)^{\frac{1}{p_0}} \,\int_0^{8r}  V(s)^{-\beta} \Big(\frac{1}{\rho(3B)} \int_{3B}|\Omega_f(z,s)|^{p_0}\di\rho(z)  \Big)^{\frac{1}{p_0}}\frac{\di s}{s}\\
&= \delta(y_0)^{\frac{1}{q} +\beta  -\frac{1}{p_0}}   V({r})^{-\frac{1}{q}+\frac{1}{p_0}}    \int_0^{8r}  V(s)^{-\beta}  M^3(|\Omega_f(\cdot,s)|^{p_0})^{\frac{1}{p_0}}(x)\frac{\di s}{s}\,. 
\end{aligned}
$$
Since $\frac{1}{q}=\frac{1}{p_0}-\beta$, we get
$$
\Omega_f^{(q)}(x,r)\lesssim V({r})^{\beta} \int_0^{8r}  V(s)^{-\beta}  M^3(|\Omega_f(\cdot,s)|^{p_0})^{\frac{1}{p_0}}(x)\frac{\di s}{s}\,,
$$
and  
$$
G^R_{\alpha,q}f(x)\leq \Big( \int_0^Rr^{-2\alpha}V({r})^{2\beta} \Big[\int_0^{8r}  V(s)^{-\beta}  M^3(|\Omega_f(\cdot,s)|^{p_0})^{\frac{1}{p_0}}\frac{\di s}{s}\Big]^2 \frac{\di r}{r}\Big)^{1/2}\,.
$$
Now arguing as in \cite[p.318]{CRTN} the statement (i) follows. 
\end{proof}

We are now ready to prove the representation formula of the Sobolev norm involving the functional $D^{{\rm loc}}_{\alpha}$.

 \proof[Proof of Theorem \ref{Salphaloc} (ii)]
 
 \smallskip
 
 { \emph{STEP I}.} We shall prove that  
$$
\|S^{{\rm loc}}_{\alpha}f\|_p\lesssim \|D^{{\rm loc}}_{\alpha}f\|_p \,.
$$
Indeed, by \eqref{pallepiccole} for every $x\in G$
$$
\begin{aligned}
S^{{\rm loc}}_{\alpha}f(x)^2&= \int_0^1   \frac{1}{u^{2\alpha}V(u)}\Big(  \int_{B_u}        |f(xy^{-1})-f(x)|       \di\rho (y) \Big)^2\frac{\di u}{u}\\
&\lesssim \int_0^1\frac{1}{u^{2\alpha}V(u)}      \int_{B_u}        |f(xy^{-1})-f(x)|^2         \di\rho (y)  \frac{\di u}{u}\\
&= \int_{|y|<1}   \Big( \int_{|y|}^1 \frac{1}{u^{2\alpha+d+1}}\di u   \Big) |f(xy^{-1})-f(x)|^2 \di\rho(y)\\
&\lesssim D^{{\rm loc}}_{\alpha}f(x)^2\,,
\end{aligned}
$$
so that $\|S^{{\rm loc}}_{\alpha}f\|_p\lesssim \|D^{{\rm loc}}_{\alpha}f\|_p$ and $\|f\|_{\alpha,p}\lesssim \|D^{{\rm loc}}_{\alpha}f\|_p+\|f\|_p$. 

\smallskip
{ \emph{STEP II}.} We shall prove that, for 
$p>2d/(d+2\alpha)$
$$
\|D^{{\rm loc}}_{\alpha}f\|_p\lesssim  \|S^{{\rm
    loc}}_{\alpha}f\|_p + \|f\|_p\,.  
$$  
By applying Proposition \ref{GR} and Theorem \ref{Salphaloc} (i) it is enough to prove that
\begin{equation}\label{StepIIa}
D_{\alpha}^{{\rm loc}}f(x)\lesssim G^{{\rm loc}}_{\alpha,2}f(x)\,,
\end{equation}
and 
\begin{equation}\label{StepIIb}
\|G^{{\rm loc}}_{\alpha}f\|_p\lesssim \|S^{{\rm loc}}_{\alpha}f\|_p+\|f\|_p\,.
\end{equation}
Indeed, we observe that $G^{{\rm loc}}_{\alpha,2} f(x)
\le G^{{\rm loc}}_{\alpha,\infty} f(x)$ 
and 
apply Proposition \ref{GR} (iii) if $\alpha>d/p$, and  Proposition
\ref{GR} (i-ii) with $q=2$ if $\alpha\le d/p$. In this latter case, we
need to assume $p>2d/(d+2\alpha)$. 
To prove \eqref{StepIIa} we argue as follows: 
$$
\begin{aligned}
D^{{\rm loc}}_{\alpha}f(x)^2&= \sum_{k=-\infty}^{-1}\int_{2^{k-1}\leq
  |y|<2^k} \frac{|f(xy^{-1})-f(x)|^2}{|y|^{2\alpha}V(|y|)}\di\rho(y)\\ 
&\lesssim \sum_{k=-\infty}^{-1}    \frac{1}{2^{2k\alpha}V(2^k)}   \int_{  |y|<2^k}  |f(xy^{-1})-f(x)|^2 \di\rho(y)\\
&\lesssim \int_0^1   \frac{1}{r^{2\alpha}V({r})}  \int_{  |y|<r}  |f(xy^{-1})-f(x)|^2 \di\rho(y) \frac{\di r}{r}\\
&\lesssim \int_0^1   \frac{1}{r^{2\alpha}V({r})}  \int_{  B(x,r)}  |f(z)-f(x)|^2 \di\lambda(z) \frac{\di r}{r}\\
&\lesssim \int_0^1   \frac{\delta(x)}{r^{2\alpha}V({r})}  \int_{  B(x,r)}  |f(z)-f_{B(x,r)}|^2 \di\rho(z) \frac{\di r}{r}\\
&\qquad +\int_0^1   \frac{1}{r^{2\alpha}V({r})}  \lambda(B(x,r))   |f(x)-f_{B(x,r)}|^2  \frac{\di r}{r}\,.
\end{aligned}
$$

We apply Lemma \ref{Omega} (iv) and the following version of Hardy's inequality: 
If $g\ge0$, $g\in L^1[0,R]$, $R>0$,  $1-p<\beta<1$ and $G(r)=\int_0^r
g(t) \di t$, then,
$$
\int_0^R \frac{1}{r^\beta} \Big( \frac{G(r)}{r} \Big)^p \di r
\lesssim  \int_0^R \frac{1}{r^\beta} g(r)^p \di r\,.
$$
We obtain
$$
\begin{aligned}
D^{{\rm loc}}_{\alpha}f(x)^2&\lesssim \int_0^1
\frac{1}{r^{2\alpha+1}}   [\Omega^2_f(x,r)]^2 {\di r}  +\int_0^1
\frac{1}{r^{2\alpha+1}}\Big[   \int_{0}^{8r}  \Omega_f^2(x,u)
\frac{\di u}{u}\Big]^2\di r\\ 
&\leq G^1_{\alpha,2}f(x)^2+\int_0^8   \frac{1}{u^{2\alpha+1}}
[\Omega_f^2(x,u)]^2\di u\\ 
&\lesssim G^8_{\alpha,2}f(x)^2\,.
\end{aligned}
$$
We shall now prove \eqref{StepIIb}. For every $B(c_B,r)$, with $r\in
(0,1]$,  $c_B\in G$  and $x\in B$ 
$$
\begin{aligned}
\int_B|f(y)-f_B|\di\rho(y)&\leq \int_B|f(y)-f(x)|\di\rho(y)+ \int_B|f(x)-f_B|\di\rho(y)\\
&\leq 2 \int_B|f(y)-f(x)|\di\rho(y)\,.
\end{aligned}
$$
Using the fact that $\rho(B(c_B,r))=\delta(c_B)^{-1}V({r})\approx \delta(x)^{-1}V({r})$, we deduce that
$$
\Omega_f(x,r)\lesssim \frac{2}{\delta(x)^{-1}V({r})}\int_{B(x,2r)} |f-f(x)|\di\rho\,,
$$
so that
$$
\begin{aligned}
\int_0^1 r^{-2\alpha-1}[\Omega_f(x,r)]^2\di r
&\lesssim  \int_0^1
r^{-2\alpha} \Big[ \delta(x) V({r})^{-1}\int_{B(x,2r)}
|f(y)-f(x)|\di\rho(y) \Big]^2\frac{\di r}{r}\\ 
&\lesssim \int_0^1 r^{-2\alpha} \Big[ \delta(x)
V({r})^{-1}\int_{|w|<2r} |f(xw)-f(x)|
\delta^{-1}(xw)\di\lambda(w)\Big]^2\frac{\di r}{r}\\ 
&= \int_0^1 r^{-2\alpha} \Big[  V({r})^{-1}\int_{|w|<2r} |f(xw)-f(x)|
\di\rho(w)\Big]^2\frac{\di r}{r}\\ 
&= \int_0^1 r^{-2\alpha} \Big[  V({r})^{-1}\int_{|w|<2r}
|f(xw^{-1})-f(x)| \delta(w)\di\rho(w)\Big]^2\frac{\di r}{r}\\ 
&\lesssim [S^2_{\alpha}f(x)]^2\,,
\end{aligned}
$$
where we used the fact that the modular function is bounded on $B_{2r}$. It follows that
$$
\|G^{{\rm loc}}_{\alpha}f\|_p\lesssim \|S^2_{\alpha}f\|_p\lesssim \|S^{{\rm loc}}_{\alpha}f\|_p+\|f\|_p\,,
$$
as required in \eqref{StepIIb}. This concludes the proof. 

\qed

 \medskip

 \section{Proof of Theorem \ref{t: productlemma2}} \label{secproof}

%%\begin{proof}
We first prove Theorem \ref{t: productlemma2} for $\alpha\in
[0,1)$. The case when $\alpha=0$ is trivial. Suppose that $\alpha\in
(0,1)$, $p_1,q_2\in (1,\infty]$ and $r,p_2,q_1\in (1,\infty)$ are such
that $\frac1r=\frac{1}{p_i}+\frac{1}{q_i}$, $i=1,2$. Take $f\in
L^{p_1}\cap L^{p_2}_{\alpha}$ and $g\in L^{q_2}\cap
L^{q_1}_{\alpha}$. According to Theorem \ref{Salphaloc}  
$$
\|fg\|_{\alpha,r}\lesssim \|S^{{\rm loc}}_{\alpha}(fg)\|_r+\|fg\|_r\,.
$$  
By H\"older's inequality one has
$$
\|fg\|_r\leq \|f\|_{p_1}\|g\|_{q_1}\leq \|f\|_{p_1}\|g\|_{\alpha,q_1}\,,
$$
and
$$
\|fg\|_r\leq \|f\|_{p_2}\|g\|_{q_2}\leq \|f\|_{\alpha,p_2}\|g\|_{q_2}\,.
$$
Moreover,
$$
\begin{aligned}
S^{{\rm loc}}_{\alpha}(fg)(x)&\leq \Big(\int_0^1\Big[ \frac{1}{u^{\alpha}
  V(u)}\int_{|y|<u}|(fg)(xy^{-1})-g(xy^{-1})f(x)|\di\rho(y)
\Big]^2\frac{\di u}{u} \Big)^{1/2}\\ 
&\qquad \qquad +\Big(\int_0^1\Big[ \frac{1}{u^{\alpha}
  V(u)}\int_{|y|<u}|f(x)g(xy^{-1})-(fg)(x)|\di\rho(y) \Big]^2\frac{\di
  u}{u} \Big)^{1/2}\\ 
&=I(x)+I\!I(x)\,.
\end{aligned}
$$
Obviously,
$$
I\!I(x)= |f(x)S^{{\rm loc}}_{\alpha}g(x)|\,,
$$
so that by H\"older's inequality
$$
\|I\!I\|_r\leq \|f\|_{p_1}\|S^{{\rm loc}}_{\alpha}g\|_{q_1}
\lesssim \|f\|_{p_1} \| g\|_{\alpha, q_1}\,.
$$
To estimate $I(x)$ we choose $p,q>1$ such that $q=p'$, $1<p<q_2$ and
$p_2>\frac{qd}{d+q\alpha}$. By H\"older's inequality we obtain 
$$
\begin{aligned}
I(x)&\lesssim  \Big(\int_0^1  \Big[ \frac{1}{V(u)}
\int_{|y|<u}|g(xy^{-1})|^p\di\rho(y) \Big]^{2/p}  \\ 
& \qquad\qquad\times \Big[ \frac{1}{V(u)}  \int_{|y|<u}|
f(xy^{-1})-f(x)  |^q\di\rho(y) \Big]^{2/q}\frac{\di u}{u^{2\alpha+1}}
\Big)^{1/2}\\ 
&=  \Big(\int_0^1  \Big[ \frac{1}{V(u)}
\int_{|y|<u}|g(xy)|^p\di\lambda(y) \Big]^{2/p}  \Big[ \frac{1}{V(u)}
\int_{|y|<u}|     f(xy)-f(x)  |^q\di\lambda(y) \Big]^{2/q}\frac{\di
  u}{u^{2\alpha+1}} \Big)^{1/2}\\ 
&=  \Big(\int_0^1  \Big[ \frac{1}{V(u)}
\int_{B(x,u)}|g(z)|^p\di\lambda(z) \Big]^{2/p}  \Big[ \frac{1}{V(u)}
\int_{B(x,u)}|     f(z)-f(x)  |^q\di\lambda(z) \Big]^{2/q}\frac{\di
  u}{u^{2\alpha+1}} \Big)^{1/2}\\ 
&\lesssim   \Big(\int_0^1  \Big[ \frac{\delta(x)}{V(u)}
\int_{B(x,u)}|g(z)|^p\di\rho(z) \Big]^{2/p}  \Big[
\frac{\delta(x)}{V(u)}  \int_{B(x,u)}|     f(z)-f(x)  |^q\di\rho(z)
\Big]^{2/q}\frac{\di u}{u^{2\alpha+1}} \Big)^{1/2}\\ 
&\lesssim \Big(M^1(|g|^p)(x)\Big)^{1/p}\,  \Big(\int_0^1  \Big[ \frac{\delta(x)}{V(u)}
\int_{B(x,u)}|     f(z)-f(x)  |^q\di\rho(z) \Big]^{2/q}\frac{\di
  u}{u^{2\alpha+1}} \Big)^{1/2}\\ 
&\leq  \Big(M^1(|g|^p)(x)\Big)^{1/p}\,G^8_{\alpha,q}f(x)\,.
\end{aligned}
$$
The last inequality follows as in \cite[p. 322-323]{CRTN}. 

Therefore, the boundedness properties of the local maximal function,
Proposition \ref{GR},  formula \eqref{StepIIb}, Theorem \ref{Salphaloc} 
(i), and H\"older's
inequality imply that 
$$
\|I\|_r\leq
\Big\|\Big(M^1(|g|^p)(x)\Big)^{1/p}\Big\|_{q_2} 
\big( \|G^{{\rm loc}}_{\alpha,q}f\|_{p_2} + \|f\|_{p_2} \big) 
\lesssim
   \|g\|_{q_2}\,\|
f\|_{\alpha,p_2}\,. 
$$
In conclusion,
$$
\|S^{{\rm loc}}_{\alpha}(fg)\|_r\lesssim  \|f\|_{p_1}\,\| g\|_{\alpha,q_1}+  \|g\|_{q_2}\,\| f\|_{\alpha,p_2}\,,
$$
as required to prove Theorem \ref{t: productlemma2} for $\alpha\in (0,1]$. 
\smallskip

We now prove Theorem \ref{t: productlemma2}  for $\alpha>1$. To do so,
we argue by induction. Suppose that the theorem holds for a certain
$\alpha>0$: we shall show that it holds for
$\beta=\alpha+1$. According to Proposition \ref{sobolevspaces}(ii)  
$$
\|fg\|_{\beta,r}\approx \|fg\|_{\alpha,r}+\sum_{i=1}^q\|X_i(fg)\|_{\alpha,r}\,.
$$
On the one hand, by the inductive hypothesis 
$$
\|fg\|_{\alpha,r}\lesssim
\|f\|_{p_1}\|g\|_{\alpha,q_1}+\|f\|_{\alpha,p_2}\|g\|_{q_2}\leq
\|f\|_{p_1}\|g\|_{\beta,q_1}+\|f\|_{\beta,p_2}\|g\|_{q_2}\,. 
$$
On the other hand, for every $i=1,\dots,q$,
$$
\|X_i(fg)\|_{\alpha,r}\leq \|(X_if)\,g\|_{\alpha,r}+\|f(X_ig)\|_{\alpha,r}\,.
$$
Using the inductive hypothesis, 
$$
\|f(X_ig)\|_{\alpha,r}\lesssim \|f\|_{\alpha,p_3}\|X_ig\|_{q_3}+\|X_ig\|_{\alpha,q_1}\|f\|_{p_1}\leq \|f\|_{\alpha,p_3}\|X_ig\|_{q_3}+\|g\|_{\beta,q_1}\|f\|_{p_1}  \,,
$$
where $\frac{1}{p_3}=\frac{\alpha}{\beta p_2}+\frac{1}{\beta p_1}$ and
$\frac{1}{q_3}=\frac{\alpha}{\beta q_2}+\frac{1}{\beta q_1}$. One
checks that $\frac{1}{p_3}+\frac{1}{q_3}=\frac{1}{r}$. By Proposition
\ref{int}  
$$
\|f\|_{\alpha,p_3}\lesssim \|f\|^{\alpha/{\beta}}_{\beta,p_2}\|f\|_{p_1}^{1-\alpha/{\beta}}\,,
$$
and by Theorem \ref{local-Riesz-trans-Lp} there exists $c$ sufficiently large such that  
$$
\|X_ig\|_{q_3}\lesssim \|(cI+\Delta)^{1/2}g\|_{q_3}\sim \|g\|_{1,q_3} \lesssim
\|g\|_{\beta,q_1}^{1/{\beta}}\|g\|_{q_2}^{1-1/{\beta}}\,,
$$
where we applied Corollary \ref{sobolevspaces}(iii) and Proposition
\ref{int}. It follows that
$$
\|f\|_{\alpha,p_3}\|X_ig\|_{q_3}\lesssim \big(
\|f\|_{\beta,p_2}\|g\|_{q_2}  \big)^{\alpha/{\beta}} \big(
\|g\|_{\beta,q_1}\|f\|_{p_1}  \big)^{1/{\beta}}\lesssim
\|f\|_{\beta,p_2}\|g\|_{q_2}+\|f\|_{p_1}\|g\|_{\beta,q_1}  \,.
$$
%In the same fashion, we obtain
%$$
%\|X_ig\|_{\alpha,q_2}\|f\|_{p_2}\lesssim \|g\|_{\beta,q_2}\|f\|_{p_2}\,.
%$$
In conclusion,
$$
\|f(X_ig)\|_{\alpha,r}\lesssim \|f\|_{\beta,p_2}\|g\|_{q_2}+\|f\|_{p_1}\|g\|_{\beta,q_1}   \,,
$$
as required. The term $\|g (X_if)\|_{\alpha,r}$ can be treated in the similar way, so
that the proof of the induction argument is complete and the theorem
is proved for every $\alpha\geq 0$.  \medskip
\qed

\section{Final remarks}  \label{secfinal} 
As we mentioned in the Introduction, we shall apply our main result Theorem
\ref{t: productlemma} to the problem of well-posedness and regularity
for solutions of the Cauchy problem for certain nonlinear differential
equations involving the subLaplacian $\Delta$ on $G$, such as the heat
and Schr\"odinger equations, see \cite{PV}.
 
We would like to point out that our results, if on one hand solve the
question of when $L^p_\alpha\cap L^\infty$ is an algebra on a generic
Lie group, on the other hand leave open several interesting questions.

First of all, given the (counter)example in Theorem
\ref{contro-esempio}, it is certainly worth investigating  the
analogous of the results in the present paper
 in the case of the weighted Lebesgue and Sobolev spaces
 $L^p_\alpha(\delta^\gamma)$.  This kind of weights arise naturally
 when considering the Sobolev embedding theorem (see \cite{V2}). 
Moreover,   
the spaces $L^p_\alpha(\delta^\gamma)$
 might turn out to be the correct spaces 
 for the well-posedness of some Cauchy problems 
--- 
see \cite{APV}, where
 Strichartz estimates involving such weighted Lebesgue spaces are
 proved for the Schr\"odinger equation 
 associated with $\Delta$ on a class of Lie groups of exponential growth.

 Finally, we 
mention that on a generic Lie group $G$, 
the $L^p$-boundedness of the Riesz transforms $\cR_j$, $j=1,\dots,q$, 
is not known, while it is known that higher order Riesz transforms
might be unbounded (see the Introduction). These problems are
connected with the study of the analogue of Theorems \ref{t:
  productlemma} and  
\ref{t:
  productlemma2}  
   for the homogeneous Sobolev spaces in our setting, which would be
   another interesting problem to investigate in the context of
   nonunimodular Lie groups.  
\mbox{ \ }

\end{document}